\documentclass[11pt]{amsart}
\usepackage{graphicx}
\usepackage{oldgerm}
\usepackage{mathdots}  
\usepackage{stmaryrd}
\usepackage{bm}
\usepackage[all]{xy}
\usepackage{color}
\usepackage{enumerate}

\usepackage{hyperref}
\usepackage{mathrsfs}
\usepackage{tikz} 
\usetikzlibrary{arrows,%
  decorations.markings,%
  matrix,%
  shapes}

\usepackage[latin1]{inputenc}
\usepackage{amsmath,amsthm, amscd, amssymb, amsfonts}

\numberwithin{equation}{section}

\theoremstyle{plain}

\numberwithin{equation}{section}

\newtheorem{theorem}[equation]{Theorem}
\newtheorem{corollary}[equation]{Corollary}
\newtheorem{proposition}[equation]{Proposition}
\newtheorem{lemma}[equation]{Lemma}

\theoremstyle{definition}
\newtheorem*{definition}{Definition}
\newtheorem*{notation}{Notation}

\theoremstyle{definition}
\newtheorem{remark}[equation]{Remark}

\newtheorem*{conjecture}{Conjecture}

\newcommand{\C}{\mathscr{C}}
\newcommand{\M}{\mathscr{M}}

\newcommand\Aut{\operatorname{Aut}}
\newcommand\Ext{\operatorname{Ext}}

\newcommand\cx{\operatorname{cx}}

\newcommand\op{\operatorname{op}}
\newcommand\ot{\otimes}
\renewcommand\mod{\operatorname{mod}}
\newcommand\Hom{\operatorname{Hom}}
\newcommand\stHom{\operatorname{\underline{Hom}}}

\newcommand\Hoch{\operatorname{HH}}

\newcommand\VC{V_{\C}}
\newcommand\VM{V_{\M}}
\newcommand\unit{\mathbf{1}}

\DeclareMathOperator{\Ker}{Ker}
\newcommand{\DOT}{\setlength{\unitlength}{1pt}\begin{picture}(2.5,2)(1,1)\put(2,3){\circle*{2}}\end{picture}}
\newcommand{\bu}{\DOT}

\newcommand{\Coh}{\operatorname{H}\nolimits}
\newcommand{\Ho}{\operatorname{\Coh^{\bu}}\nolimits}
\newcommand{\Maxspec}{\operatorname{MaxSpec}\nolimits}
\newcommand{\Spec}{\operatorname{Spec}\nolimits}
\newcommand{\Supp}{\operatorname{Supp}\nolimits}

\newcommand{\m}{\mathfrak{m}}
\newcommand{\p}{\mathfrak{p}}
\newcommand{\q}{\mathfrak{q}}

\makeatletter
\def\blx@maxline{77}
\makeatother

\begin{document}
\title[The tensor product property]
{Support varieties for finite tensor categories: the tensor product property}

\author{Petter Andreas Bergh, Julia Yael Plavnik, Sarah Witherspoon}

\address{Petter Andreas Bergh \\ Institutt for matematiske fag \\
NTNU \\ N-7491 Trondheim \\ Norway} \email{petter.bergh@ntnu.no}
\address{Julia Yael Plavnik \\ Department of Mathematics \\ Indiana University \\ Bloomington \\ Indiana 47405 \\ USA\\ \& Fachbereich Mathematik\\ Universit\"at Hamburg\\ Hamburg  20146\\ Germany}
\email{jplavnik@iu.edu}
\address{Sarah Witherspoon \\ Department of Mathematics \\ Texas A \& M University \\ College Station \\ Texas 77843 \\ USA}
\email{sjw@tamu.edu}
\subjclass[2020]{16E40, 16T05, 18M05, 18M15}
\keywords{finite tensor categories; support varieties; tensor product property}
\date{5 June 2024}     

\begin{abstract}
We show that in a finite tensor category, the tensor product property holds for support varieties if and only if 
it holds between indecomposable periodic objects.
We apply this result
to deduce the tensor product property for a large class of categories,
those of modules for skew group algebras formed by
exterior algebras with certain finite group actions.
These include 
the symmetric finite tensor categories over algebraically closed 
fields of characteristic zero, thus giving a new proof of 
the tensor product property for these categories. 
\end{abstract}

\maketitle

\section{Introduction}

Given a finite tensor category $\C$, one can attach a support variety $\VC (X)$ to each object $X$, using the spectrum of the cohomology ring. It has been conjectured by Etingof and Ostrik that every finite tensor category has finitely generated cohomology. As shown in our paper~\cite{BPW}, whenever this holds, the support varieties encode homological properties of the objects, in much the same way as do cohomological support varieties over group algebras, more general cocommutative Hopf algebras, and commutative complete intersection rings.
A parallel development of support varieties for module categories over
tensor triangulated categories, with complementary results, is given
by Buan, Krause, Snashall, and Solberg~\cite{BKSS}. 

When does the tensor product property hold for support varieties? That is, what conditions - if any - will guarantee that
$$\VC(X \ot Y) = \VC(X) \cap \VC(Y)$$
for all objects $X,Y \in \C$? This property always holds for support varieties over group algebras of finite groups and more generally finite dimensional cocommutative Hopf algebras~\cite{FriedlanderPevtsova} as well as a number of other 
Hopf algebras (e.g.~\cite{NVY,NegronPevtsova,PevtsovaWitherspoon2}). 
There are classes of examples for which it does not hold~\cite{BensonWitherspoon,BPW2,PlavnikWitherspoon}.
One reason why one would seek such a property is to classify thick tensor ideals in the stable category; the tensor product
property is used in classifying such thick tensor ideals in a number of settings.
Nakano, Vashaw, and Yakimov~\cite{NVY} proposed a modified version of
this tensor product property that may also be used, and is known to hold
more generally, but we will not pursue that direction here.

Our main theorem (Theorem~\ref{thm:main}) states
that when $\C$ is braided, the tensor product property holds for all objects if and only if it holds between \emph{indecomposable periodic} objects. In other words, we show that if $\VC(X \ot Y) = \VC(X) \cap \VC(Y)$ for all indecomposable periodic $X,Y \in \C$, then the tensor product property holds for all objects. Thus the question of whether the tensor product property holds reduces to indecomposable periodic objects, or, equivalently, to indecomposable objects of complexity one. We prove this reduction in the more general setting of a module category over $\C$.
To the best of our knowledge, this reduction to complexity one is new;
we are not aware of such an approach to prove a tensor product property 
in any setting in the literature.

A refinement of our main theorem is Theorem~\ref{thm:mainalternative},
stating that it is sufficient to show
for each pair of indecomposable periodic objects whose support varieties coincide,
that their tensor product (or module product) is not projective.
In some settings, there are representation theoretic tools
strong enough to check this condition directly. 

We illustrate the utility of our main theorems by verifying the
tensor product property 
for some braided categories of modules over skew group algebras.
This involves a careful comparison to modules for a subalgebra 
and some periodic cyclic modules having irreducible support varieties. 
This method is reminiscent of rank varieties, in particular those 
in~\cite{AAH,BensonErdmannHolloway, BerghErdmann, PevtsovaWitherspoon},
and indeed the theory of rank varieties could be developed further
to apply here.
We choose instead to develop only what is needed
to demonstrate the tensor product property for these examples
as a consequence of our main theorems.
As a special case, we give a new approach to the tensor
product property for any  
symmetric finite tensor category 
over an algebraically closed field of characteristic zero,
relying on Deligne's classification: 
these categories are equivalent to those of finitely generated modules 
of certain skew group algebras over exterior algebras.
We thus recover a result of  
Drupieski and Kujawa~\cite{DK}, namely
the tensor product property
for finite dimensional cocommutative Hopf superalgebras in characteristic~0. 
By contrast, 
our result is largely orthogonal to
that of Benson, Iyengar, Krause, and Pevtsova~\cite{BIKP}, 
the tensor product property for unipotent 
finite dimensional cocommutative Hopf superalgebras in odd characteristic,
since our skew group algebras are typically not unipotent
and we assume the characteristic does not divide the order of the group.

\subsection*{Acknowledgments}
P.\ A.\ Bergh would like to thank the organizers of the Representation Theory program hosted by the Centre for Advanced Study at The Norwegian Academy of Science and Letters, where he spent parts of fall 2022. J.Y.\ Plavnik was partially supported by NSF grant DMS-2146392 and by Simons Foundation Award 889000 as part of the Simons Collaboration on Global Categorical Symmetries. J.P. would like to thank the hospitality
and excellent working conditions at the Department of Mathematics at Universit\"at Hamburg, where she has carried out part of this research as an Experienced Fellow of the Alexander von Humboldt Foundation. S.J.\ Witherspoon was partially supported by NSF grants 1665286 and 2001163.

\section{Preliminaries}\label{sec:prelim}

Let us fix from the very beginning the categories that we will be working with throughout the whole paper. We follow the definitions and conventions from the book \cite{EGNO}.

\begin{notation}
We fix a field $k$ -- not necessarily algebraically closed -- together with a finite tensor $k$-category $\left ( \C, \ot, \unit \right )$ and an exact left module category $\left ( \M, \ast \right )$ over $\C$. Furthermore, we make the assumptions that $\M$ has a finite set of isomorphism classes of simple objects.
\end{notation}

Thus $\C$ and $\M$ are both locally finite $k$-linear abelian categories, and $\C$ has a finite set of isomorphism classes of simple objects, each of which admits a projective cover. Moreover, there are associative (up to functorial isomorphisms) bifunctors 
$$\ot: \C\times \C \rightarrow \C \hspace{3mm} \text{and}  \hspace{3mm} \ast: \C\times \M \rightarrow \M$$ 
called the tensor product and the module product, which are compatible with the abelian structures of the categories, together with a unit object $\unit \in \C$ (with respect to both $\ot$ and $\ast$) which is simple as an object of $\C$. Furthermore, the bifunctor $\ast$ is exact in the first argument, and whenever $P$ is a projective object in $\C$, then $P \ast M$ is projective in $\M$ for all $M \in \M$. Finally, the category $\C$ is rigid, meaning that all objects have left and right duals.

\begin{remark}\label{rem:properties}
(1) Since $\C$ is rigid, the tensor product $\ot$ is biexact, by \cite[Proposition 4.2.1]{EGNO}. Moreover, by \cite[Proposition 4.2.12]{EGNO}, the collection of projective objects forms an ideal of $\C$; the tensor product between a projective object and any other object is again projective. 

(2) It follows from \cite[Proposition 7.1.6]{EGNO} that the bifunctor $\ast$ is also exact in the second argument (and hence biexact), and that whenever $Q$ is a projective object in $\M$, then $X \ast Q$ is projective in $\M$ for all $X \in \C$.

(3) Since we have assumed that the module category $\M$ also has a finite set of isomorphism classes of simple objects, this category is finite, like $\C$. This implies that all the objects of both $\C$ and $\M$ admit projective covers. Namely, as explained in \cite[Section 1.8]{EGNO}, each of the categories is equivalent to the category of finitely generated left modules over some finite dimensional $k$-algebra. Using projective covers, we can construct a minimal projective resolution for any given object, and this resolution is unique up to isomorphism. 

(4) By \cite[Corollary 7.6.4]{EGNO}, both $\C$ and $\M$ are quasi-Frobenius, that is, the projective objects are precisely the injective objects.

(5) The category $\C$ is trivially a left module category over itself, with the tensor product as the module product. Therefore everything we develop and prove for objects of $\M$ holds for objects of $\C$.

(6) Since the unit object $\unit$ is simple, the $k$-algebra $\Hom_{\C}( \unit, \unit )$ is a division ring, that is, all the nonzero elements are invertible. This ring is in fact commutative (see the paragraphs following this remark), and therefore a finite field extension of $k$. In particular, when $k$ is algebraically closed, then $\Hom_{\C}( \unit, \unit ) = k$.

(7) We refer to \cite[Section 2]{BPW} for an overview of some of the homological properties and techniques for finite tensor categories that we use throughout. Almost all the results and concepts carry over to $\M$ as well.
\end{remark}

There are many important examples of tensor categories in which the tensor product is not commutative. However, in our main results, we need this property, both the standard and a stronger version. The tensor category $\C$ is called {\em braided} if for all objects $X,Y \in \C$, there are functorial isomorphisms 
$$X\ot Y \xrightarrow{b_{X,Y}} Y\ot X$$
that satisfy the hexagonal identities defined in \cite[Definition 8.1.1]{EGNO}. If, in addition, these braiding isomorphisms satisfy
$$b_{Y,X} \circ b_{X,Y} = 1_{X\ot Y}$$
for all objects $X$ and $Y$, then $\C$ is {\em symmetric}. An example of the latter is the category of finitely generated left modules over a group algebra. However, in general, if $H$ is a finite dimensional Hopf algebra, then the category $\mod H$ of finitely generated left $H$-modules is a finite tensor category that is not necessarily braided.
 
Given objects $M,N \in \M$, we denote by $\Ext_{\M}^*(M,N)$ the graded $k$-vector space $\oplus_{n=0}^{\infty} \Ext_{\M}^n(M,N)$. The module product $- \ast M$ induces a homomorphism
$$\Ext_{\C}^*( \unit, \unit ) \xrightarrow{\varphi_M} \Ext_{\M}^*(M,M)$$
of graded $k$-algebras, making $\Ext_{\M}^*(M,N)$ both a left and a right module over the cohomology algebra $\Ext_{\C}^*( \unit, \unit )$, via $\varphi_N$ and $\varphi_M$ followed by Yoneda composition. In particular, for objects $X,Y \in \C$, the left and right scalar actions of $\Ext_{\C}^*( \unit, \unit )$ on $\Ext_{\C}^*(X,Y)$ are induced by the tensor products $- \ot Y$ and $- \ot X$, respectively, followed by Yoneda composition. However, not only is the algebra $\Ext_{\C}^*( \unit, \unit )$ graded-commutative by \cite[Theorem 1.7]{SA}, the following lemma and its corollary show that for objects $M,N \in \M$, the left and the right scalar actions of $\Ext_{\C}^*( \unit, \unit )$ on $\Ext_{\M}^*(M,N)$ coincide up to a sign, when we only consider homogeneous elements. The proof is a straightforward adaptation of the proof of \cite[Theorem 1.1]{SS04}, and we omit it for brevity. We use the symbol $\circ$ to denote Yoneda composition, as well as ordinary composition of maps.

\begin{lemma}\label{lem:leftright}
Given any objects $X,Y \in \C$, $M,N \in \M$, integers $m,n \ge 0$ and elements $\eta \in \Ext_{\C}^m(X,Y)$ and $\theta \in \Ext_{\M}^n(M,N)$, the equality
$$( \eta \ast N ) \circ ( X \ast \theta ) = (-1)^{mn} ( Y \ast \theta ) \circ ( \eta \ast M )$$
holds in $\Ext_{\M}^{m+n}( X \ast M, Y \ast N )$.
\end{lemma}

Specializing to the case when $X = Y = \unit$, we obtain what we are after, recorded in the following corollary. Note also that when we specialize even further, by taking $\M = \C$ and $M=N= \unit$, we recover the graded-commutativity of $\Ext_{\C}^*( \unit, \unit )$.

\begin{corollary}\label{lem:centralaction}
Given any objects $M,N \in \M$ and elements $\eta \in \Ext_{\C}^m( \unit, \unit )$ and $\theta \in \Ext_{\M}^n(M,N)$, the equality
$$\eta \cdot \theta = (-1)^{mn} \theta \cdot \eta$$
holds.
\end{corollary}

The algebra $\Ext_{\C}^*( \unit, \unit )$ is the \emph{cohomology ring} $\Coh^*( \C )$ of $\C$. It is at the center of the following conjecture from \cite{EO}, a conjecture which is still open:
\begin{conjecture}
\sloppy The cohomology ring $\Coh^*( \C )$ is finitely generated, and $\Ext_{\C}^*(X,X)$ is a finitely generated $\Coh^*( \C )$-module for all objects $X \in \C$.
\end{conjecture}

If the characteristic of the ground field $k$ is two, then graded-commutativity is the same as ordinary commutativity. If, on the other hand, the characteristic of $k$ is not two, then the even part of the cohomology ring $\Coh^*( \C )$ is  commutative, and the homogenous elements of odd degrees square to zero. When we work with support varieties, nilpotent elements in the ambient commutative ring are redundant, and this motivates the first part of the following definition.

\begin{definition}
(1) We define
$$\Ho ( \C ) = \left \{ 
\begin{array}{ll}
\Coh^*( \C ) & \text{if the characteristic of $k$ is two,} \\
\Coh^{2*}( \C ) & \text{if not.}
\end{array} 
\right.$$

(2) We say that the finite tensor category $\C$ satisfies the \emph{finiteness condition} \textbf{Fg} if the cohomology ring $\Coh^*( \C )$ is finitely generated, and $\Ext_{\C}^*(X,X)$ is a finitely generated $\Coh^*( \C )$-module for all objects $X \in \C$.
\end{definition}

As explained in \cite[Remark 3.5]{BPW}, the finiteness condition \textbf{Fg} and the conjecture can be stated in terms of $\Ho ( \C )$ instead of $\Coh^*( \C )$. Namely, the condition \textbf{Fg} holds for $\C$ if and only if $\Ho ( \C )$ is finitely generated, and $\Ext_{\C}^*(X,X)$ is a finitely generated $\Ho ( \C )$-module for every object $X \in \C$. Note also that when \textbf{Fg} holds for $\C$, then for all objects $X,Y \in \C$, the $\Coh^* ( \C )$-module $\Ext_{\C}^*(X,Y)$ is finitely generated, and not just the two modules $\Ext_{\C}^*(X,X)$ and $\Ext_{\C}^*(Y,Y)$. This follows from the simple fact that the $\Coh^* ( \C )$-module $\Ext_{\C}^*(X \oplus Y,X \oplus Y)$ is finitely generated by assumption, and it has $\Ext_{\C}^*(X,Y)$ as a direct summand. 

\begin{remark}\label{rem:FGexact}
When the finiteness condition \textbf{Fg} holds for $\C$, then what about the cohomology of $\M$? It turns out that it is automatically finitely generated. Namely, by \cite[Proposition 3.5]{NP}, if \textbf{Fg} holds for $\C$, then $\Ext_{\M}^*(M,M)$ is a finitely generated $\Coh^*( \C )$-module for every object $M \in \M$. As for $\C$, this implies that for all objects $M,N \in \M$, the $\Coh^* ( \C )$-module $\Ext_{\M}^*(M,N)$ is finitely generated. Moreover, also here we may replace $\Coh^*( \C )$ with $\Ho ( \C )$.
\end{remark}

For objects $M,N \in \M$, we now define
$$I_{\M}(M,N) = \left \{ \eta \in \Ho ( \C ) \mid \eta \cdot \theta = 0 \text{ for all } \theta \in \Ext_{\M}^*(M,N) \right \} , $$
that is, the annihilator ideal of $\Ext_{\M}^*(M,N)$ in $\Ho ( \C )$. For a single object $M$ we write just $I_{\M}(M)$ instead of $I_{\M}(M,M)$. Moreover, for any ideal $I \subseteq \Ho ( \C )$, we denote by $Z(I)$ the set of maximal ideals $\m \in \Ho ( \C )$ with $I \subseteq \m$. Finally, we set $\m_0 =  \Coh^{+} (\C)$, the ideal generated by all the homogeneous elements of positive degrees in $\Ho ( \C )$. Then $\m_0$ is the unique graded maximal ideal of $\Ho ( \C )$, since $\Coh^{0} (\C)$ is a field; see Remark \ref{rem:properties}(6). Consequently, the annihilator ideal that we just defined, which is graded, must be contained in $\m_0$ whenever $\Ext_{\M}^*(M,N)$ is nonzero. 

\begin{definition}
The \emph{support variety} of an ordered pair of objects $(M,N)$ in $\M$ is
$$\VM(M,N) \stackrel{\text{def}}{=} \{ \m_0 \} \cup Z ( I_{\M} (M,N) )$$
For a single object $M \in \M$, we define its support variety as $\VM(M) = \VM(M,M)$.
\end{definition} 

In the definition, the explicit inclusion of the unique graded maximal ideal $\m_0$ has been made in order to avoid empty support varieties; if $\Ext_{\M}^*(M,N)$ is nonzero, then $\m_0$ is automatically contained in the set $Z ( I_{\M} (M,N) )$, since $I_{\M} (M,N)$ is a graded proper ideal of $\Ho ( \C )$. The support variety $\VM(M,N)$ is called \emph{trivial} if $\VM(M,N) = \{ \m_0 \}$.

\begin{remark}\label{rem:trivialvariety}
(1) When we deal with objects in the category $\C$ itself, we use the notation $I_{\C}(X,Y)$, $\VC(X,Y)$ and $\VC(X)$.

(2) We define $\VC$ as $\VC ( \unit )$; this is just the set of maximal ideals of the cohomology ring $\Ho ( \C )$. Note that $\VM(M,N) \subseteq \VC$ for all $M,N \in \M$.

(3) An important feature of support varieties -- probably the most important -- is the dimension. For objects $M,N \in \M$, the dimension of $\VM(M,N)$, denoted $\dim \VM(M,N)$, is defined to be the Krull dimension of the ring $\Ho ( \C ) / I_{\M}(M,N)$. If this dimension is zero, then the support variety is necessarily trivial. For suppose that $\VM(M,N)$ contains a maximal ideal $\m$ other than $\m_0$, and let $\m^*$ be the graded ideal of $\Ho ( \C )$ generated by all the homogeneous elements in $\m$. By \cite[Lemma 1.5.6]{BrunsHerzog}, this is a prime ideal, and so since the graded ideal $I_{\M} (M,N)$ is contained in $\m$, we see that $I_{\M} (M,N) \subseteq \m^*$ (for $\m^*$ is the unique maximal graded ideal contained in $\m$). As $\m$ is not graded, the inclusion $\m^* \subset \m$ is strict, hence the Krull dimension of $\Ho ( \C ) / I_{\M}(M,N)$ is at least $1$. Thus when $\dim \VM(M,N) = 0$, then $\VM(M,N) = \{ \m_0 \}$. However, when the finiteness condition \textbf{Fg} holds for $\M$, then the converse is also true, so that
$$\dim \VM(M,N) = 0 \Longleftrightarrow \VM(M,N) = \{ \m_0 \}$$
For in this case, if $\VM(M,N) = \{ \m_0 \}$, then if $\Ext_{\M}^*(M,N)$ is nonzero, the radical $\sqrt{I_{\M}(M,N)}$ equals $\m_0$, by \cite[Theorem 25]{Matsumura}. Consequently, the Krull dimension of $\Ho ( \C ) / I_{\M}(M,N)$ must be zero.
\end{remark}

In the following result, we collect some of the basic properties enjoyed by support varieties for objects of $\M$. For objects of $\C$, these properties were listed in \cite[Proposition 3.3]{BPW}. In that paper, we made the assumption that the ground field $k$ is algebraically closed, but that assumption was never needed. The proofs carry over to the general setting of exact module categories; some of them rely on Corollary \ref{lem:centralaction}. Only one of the properties, number (6) below, requires an argument that is special to $\M$.

\begin{proposition}\label{prop:elementaryproperties}
For objects $M, N \in \M$, the following properties hold.

\emph{(1)} $\VM(M \oplus N) = \VM(M) \cup \VM(N)$.

\emph{(2)} $\VM(M,N) \subseteq \VM(M) \cap \VM(N)$.

\emph{(3)} $\VM(M) = \cup_{i=1}^t \VM(M,S_i) = \cup_{i=1}^t \VM(S_i,M)$, where $S_1, \dots, S_t$ are all the simple objects of $\M$ \emph{(}up to isomorphism\emph{)}.

\emph{(4)} Given any short exact sequence
$$0 \to L_1 \to L_2 \to L_3 \to 0$$
in $\M$, the inclusion $\VM(L_u) \subseteq \VM(L_v) \cup \VM(L_w)$ holds whenever $\{ u,v,w \} = \{ 1,2,3 \}$.

\emph{(5)} If there is a short exact sequence
$$0 \to M \to P \to N \to 0$$
in $\M$, in which $P$ is projective, then $\VM(M) = \VM(N)$.

\emph{(6)} For every object $X \in \C$, the inclusion $\VM(X \ast M) \subseteq \VC(X)$ holds. Moreover, if the category $\C$ is braided, then $\VM(X \ast M) \subseteq \VC(X) \cap \VM(M)$.
\end{proposition}

\begin{proof}
As mentioned, only (6) needs an argument, since the proof of \cite[Proposition 3.3]{BPW} works for the rest.

The scalar action of $\Ho ( \C )$ on $\Ext_{\M}^*(X \ast M,X \ast M)$ is defined in terms of the ring homomorphism 
$$\Ho ( \C ) \xrightarrow{\varphi_{X \ast M}} \Ext_{\M}^*(X \ast M,X \ast M)$$
which in turn is induced by the module product $- \ast (X \ast M)$. Now, for every object $Y \in \C$ there is an isomorphism $Y \ast (X \ast M) \simeq  (Y \ot X) \ast M$, functorial in $Y$. Therefore, the ring homomorphism factors as the composition
$$\Ho ( \C ) \xrightarrow{\varphi_{X}} \Ext_{\C}^*(X,X) \xrightarrow{- \ast M} \Ext_{\M}^*(X \ast M,X \ast M)$$
where the ring homomorphism $\varphi_X$ is induced by the tensor product $- \ot X$. This implies $I_{\C}(X) \subseteq I_{\M}(X \ast M)$, and so the inclusion $\VM(X \ast M) \subseteq \VC(X)$ follows by definition of support varieties.

Suppose now that the category $\C$ is braided, and take any homogeneous element $\eta \in \Ho ( \C )$. By definition, for every object $Y \in \C$ there is an isomorphism $Y \ot X \to X \ot Y$, functorial in $Y$, giving
$$\varphi_{X \ast M} ( \eta ) = \eta \ast (X \ast M) = ( \eta \ot X ) \ast M = (X \ot \eta ) \ast M = X \ast ( \eta \ast M ) = X \ast \varphi_M ( \eta )$$
as elements of $\Ext_{\M}^*(X \ast M,X \ast M)$. Thus $I_{\M}(M) \subseteq I_{\M}(X \ast M)$, and the containment $\VM(X \ast M) \subseteq \VM(M)$ follows.
\end{proof}

Recall from Remark \ref{rem:properties}(3) that every object $M \in \M$ (and every object of $\C$) admits a minimal projective resolution $(P_{\bu}, d_{\bu})$, which is unique up to isomorphism. We define the $n$th \emph{syzygy} of $M$ to be the image of the morphism $d_n$, and denote it by $\Omega_{\M}^n(M)$ (or $\Omega_{\C}^n(X)$ for an object $X \in \C$).  As shown in \cite[Lemma 2.4]{BPW}, the minimal projective resolution has the property that
$$\Ext_{\M}^n(M,S) \simeq \Hom_{\M}(P_n,S) \simeq \Hom_{\M}(\Omega_{\M}^n(M),S)$$
for every $n \ge 1$ and every simple object $S \in \M$. 

Given a sequence $a_{\bu} = (a_0,a_1,a_2, \dots )$ of nonnegative real numbers, we denote by $\gamma (a_{\bu})$ its polynomial rate of growth, that is, the infimum of integers $c \ge 0$ for which there exists a number $b \in \mathbb{R}$ with $a_n \le bn^{c-1}$ for all $n \ge 0$. We now define the \emph{complexity} of the object $M$, denoted $\cx_{\M}(M)$, to be $\gamma ( \ell P_{\bu} )$, where $\ell P_n$ denotes the length of the object $P_n$. For objects of $\C$, this is not the same definition as used in \cite{BPW}, where we defined the complexity to be the rate of growth of the Frobenius-Perron dimensions of the objects of the minimal projective resolution. However, the two definitions are equivalent; there are only finitely many indecomposable projective objects of $\C$ (one for each simple object), and the two definitions just rely on attaching different sets of positive real numbers -- all at least $1$ -- to these. As explained in \cite[Remark 4.2]{BPW}, the complexity of $M$ is the same as the rate of growth of the sequence $( \dim_k \Ext_{\M}^*(M, S_1 \oplus \cdots \oplus S_t))$, where $S_1, \dots, S_t$ are all the simple objects of $\M$. Moreover, it also equals the rate of growth of the sequence whose $n$th term is the number of indecomposable summands of $P_n$.

We end this section with a result which sums up the properties that were proved in \cite{BPW} for support varieties when \textbf{Fg} holds. These properties were proved for objects in a finite tensor category, and not objects in a module category, but as for Proposition \ref{prop:elementaryproperties}, the proofs carry over to the more general setting, so we omit them. Moreover, as mentioned before Proposition \ref{prop:elementaryproperties}, the assumption we made in \cite{BPW} that $k$ be algebraically closed was never needed.

Recall first that if $\zeta$ is a nonzero homogeneous element of $\Ho ( \C )$, say of degree $n$, then it can be represented by an epimorphism $f_{\zeta} \colon \Omega_{\C}^n( \unit ) \to \unit$ (it is necessarily an epimorphism since the unit object is simple in $\C$). We denote the kernel of this morphism by $L_{\zeta}$; this object is known as \emph{Carlson's $L_{\zeta}$-object}. The module version of \cite[Theorem 5.2]{BPW} gives an inclusion $\VM( L_{\zeta} \ast M ) \subseteq Z( \zeta ) \cap \VM (M)$ for every object $M \in \M$, even without assuming that $\C$ is braided, as in Proposition \ref{prop:elementaryproperties}(6).

\begin{theorem}\label{thm:fgproperties}
If $\C$ satisfies \emph{\textbf{Fg}}, then the following hold for every object $M \in \M$.

\emph{(1)} $\cx_{\M}(M) = \gamma \left ( \dim_k \Ext_{\M}^*(M,M) \right ) = \dim \VM(M) \le \dim \Ho ( \C )$, where $\dim \Ho ( \C )$ is the Krull dimension of the cohomology ring $\Ho ( \C )$.

\emph{(2)} The object $M$ is projective if and only if $\VM (M)$ is trivial, and if and only if $\cx_{\M}(M) =0$.

\emph{(3)} $\VM( L_{\zeta} \ast M ) = Z( \zeta ) \cap \VM (M)$ for every nonzero homogeneous element $\zeta \in \Ho ( \C )$.

\emph{(4)} Given any nonempty conical subvariety $V \subseteq \VM(M)$, there exists an object $N \in \M$ with $\VM(N) = V$.

\emph{(5)} Given any nonnegative integer $c \le \cx_{\M}(M)$, there exists an object $N \in \M$ with $\cx_{\M}(N) =c$.

\emph{(6)} If $\cx_{\M}(M) \ge 1$, then there exists a short exact sequence
$$0 \to M \to K \to \Omega_{\M}^n(M) \to 0$$
for some $n \ge 0$ and some object $K \in \M$ with $\cx_{\M}(K) = \cx_{\M}(M) -1$.

\emph{(7)} Given any object $N \in \M$, the support variety $\VM(M,N)$ is trivial if and only if $\Ext_{\M}^n(M,N)=0$ for all $n \gg 0$, if and only if $\Ext_{\M}^n(M,N)=0$ for all $n \ge 1$.

\emph{(8)} If $\VM(M) = V_1 \cup V_2$ for conical subvarieties $V_1,V_2$ with $V_1 \cap V_2 = \{ \m_0 \}$, then $M \simeq M_1 \oplus M_2$ for some objects $M_1,M_2$ with $\VM(M_i) = V_i$.
\end{theorem}

\section{The module product property}\label{sec:tensor}

Recall that we have fixed a field $k$ -- not necessarily algebraically closed -- together with a finite tensor $k$-category $\left ( \C, \ot, \unit \right )$ and an exact left module category $\left ( \M, \ast \right )$ over $\C$. Moreover, we have assumed that $\M$ has a finite set of isomorphism classes of simple objects.

In this section we prove the main result (Theorem~\ref{thm:main}): the question of whether the module product property holds for support varieties reduces to the question of whether it holds if we only consider indecomposable periodic objects. By a \emph{periodic} object, we mean an object $M \in \M$ with $M \simeq \Omega_{\M}^n(M)$ for some $n \ge 1$. In other words, the minimal projective resolution of $M$ is periodic of period $n$.

We start with the following result, which, together with its corollary, characterizes the indecomposable periodic objects in terms of their complexities. 

\begin{theorem}\label{thm:complexityone}
If $\C$ satisfies \emph{\textbf{Fg}}, then the following are equivalent for an object $M \in \M$:
\begin{itemize}
\item[(1)] $\cx_{\M}(M) = 1$;
\item[(2)] $\dim \VM(M) = 1$;
\item[(3)] $M$ is isomorphic to $N \oplus Q$ for some nonzero periodic object $N$ and projective object $Q$.
\end{itemize}
\end{theorem}

\begin{proof}
The equivalence of (1) and (2) is a special case of Theorem \ref{thm:fgproperties}(1). If (3) holds, then the sequence
$$\left ( \dim_k \Ext_{\M}^n(M,M) \right )_{n=0}^{\infty}$$
is bounded and not eventually zero, and so its rate of growth is $1$. By Theorem \ref{thm:fgproperties}(1) again, this rate of growth equals the complexity of $M$, hence (1) follows.

Finally, suppose that (1) holds. Then by Theorem \ref{thm:fgproperties}(6), there exists a short exact sequence
$$0 \to M \to K \to \Omega_{\M}^n(M) \to 0$$
for some $n \ge 0$, with $\cx_{\M}(K) = 0$. By Theorem \ref{thm:fgproperties}(2), the object $K$ is then projective, and so by Schanuel's Lemma for abelian categories (see \cite[Lemma 2.2]{BPW}), there is an isomorphism $M \simeq \Omega_{\M}^{n+1}(M) \oplus Q$ for some projective object $Q$. Now take $N =  \Omega_{\M}^{n+1}(M)$; this is a periodic object since
$$N = \Omega_{\M}^{n+1}(M) \simeq \Omega_{\M}^{n+1} \left ( \Omega_{\M}^{n+1}(M) \oplus Q \right ) =  \Omega_{\M}^{n+1}(N \oplus Q) \simeq \Omega_{\M}^{n+1}(N)$$
This shows that (1) implies (3).
\end{proof}

\begin{corollary}\label{cor:periodic}
If $\C$ satisfies \emph{\textbf{Fg}}, and $M$ is a nonzero indecomposable object of $\M$, then $\cx_{\M}(M) =1$ if and only if $M$ is periodic.
\end{corollary}

We have defined support varieties in terms of the maximal ideal spectrum of $\Ho (\C)$. However, in some of the arguments that follow, we need to consider prime ideals in general; as usual, we denote the set of prime ideals of $\Ho (\C)$ by $\Spec \Ho (\C)$. For an object $M \in \M$, we denote the support of the $\Ho (\C)$-module $ \Ext_{\M}^*(M,M)$ by $\Supp_{\M} (M)$, that is, the set of prime ideals of $\Ho (\C)$ with $\Ext_{\M}^*(M,M)_{\p} \neq 0$. When the finiteness condition \textbf{Fg} holds, then 
$$\Supp_{\M} (M) = \{ \p \in \Spec \Ho (\C) \mid I_{\M}(M) \subseteq \p \}$$
and hence
$$\VM(M) = \Supp_{\M} (M) \cap \Maxspec \Ho (\C)$$
whenever $M$ is a nonzero object. Note also that the finiteness condition implies that the set of minimal primes of $\Supp_{\M} (M)$ is finite, and that these are associated primes of the $\Ho (\C)$-module $ \Ext_{\M}^*(M,M)$; see \cite[Theorem 3.1.a]{Eisenbud}. Furthermore, by \cite[Lemma 1.5.6]{BrunsHerzog}, these minimal primes are in fact graded. When $M$ is nonzero and $\p_1, \dots, \p_t$ are the minimal primes of $\Supp_{\M} (M)$, then $\VM(M) = Z( \p_1 ) \cup \cdots \cup Z( \p_t )$. The $Z( \p_i )$ are the irreducible components of $\VM(M)$, hence the support variety is irreducible if and only if $\Supp_{\M} (M)$ contains a unique minimal prime (when $i \neq j$ then $Z( \p_i ) \neq Z( \p_j )$; see the paragraphs following \cite[Remark 3.5]{BPW}). The following result shows that the support variety of an indecomposable periodic object is of this form.

\begin{proposition}\label{prop:indecperiodic}
Suppose that $\C$ satisfies \emph{\textbf{Fg}}, and that $X \in \C$ and $M \in \M$ are nonzero indecomposable periodic objects. Then $\VC(X)$ and $\VM(M)$ are irreducible \emph{(}that is, $\Supp_{\C}(X)$ contains a unique minimal prime, and so does $\Supp_{\M}(M)$\emph{)}, and either $\VC(X) = \VM(M)$, or $\VC(X) \cap \VM(M) = \{ \m_0 \}$.
\end{proposition}

\begin{proof}
Let $\p_1, \dots, \p_t$ be the minimal primes of $\Supp_{\M}(M)$, so that $\VM(M) = Z( \p_1 ) \cup \cdots \cup Z( \p_t )$; recall that these primes are all graded. Note first that none of them can be maximal, that is, equal to $\m_0$ (there is only one graded maximal ideal in $\Ho (\C)$, namely $\m_0$). For suppose this were the case, with, say, $\p_1 = \m_0$. If $t=1$, then $\VM(M) = Z( \m_0 ) = \{ \m_0 \}$, and hence $\dim \VM(M) = 0$. But $M$ is nonzero and periodic, and so by Theorem \ref{thm:complexityone}, the dimension of $\VM(M)$ must be $1$. If $\p_1 = \m_0$ and $t \ge 2$, then the prime $\p_1$ is not minimal in $\Supp_{\M}(M)$, since the other primes $\p_2, \dots, \p_t$ are contained in $\m_0$.

Thus none of the minimal primes $\p_i$ are maximal, and so the Krull dimension of $\Ho (\C) / \p_i$ is at least $1$, for each $i$. But since $\dim \VM(M) = 1$, and this dimension is the maximum among $\dim \Ho (\C) / \p_1, \dots, \dim \Ho (\C) / \p_t$, it follows that $\dim \Ho (\C) / \p_i =1$ for each $i$. Of course, since $\p_i$ is graded, the graded maximal ideal $\m_0$ belongs to $Z( \p_i )$, but this irreducible component must also contain a non-graded maximal ideal. For if $\m_0$ were the only maximal ideal containing $\p_i$, then by \cite[Theorem 25]{Matsumura} the radical $\sqrt{\p_i}$ of $\p_i$ would be $\m_0$, a contradiction since $\dim \Ho (\C) / \sqrt{\p_i} = \dim \Ho (\C) / \p_i =1$.

Suppose now that $t \ge 2$, and set $V_1 = Z( \p_1 )$ and $V_2 = Z( \p_2 ) \cup \cdots \cup Z( \p_t ) = Z( \p_2 \cdots \p_t )$. If the intersection $V_1 \cap V_2$ contains a non-graded maximal ideal $\m$, then both $\p_1$ and $\p_i$ are contained in $\m$, for some $i \ge 2$. Now consider the graded ideal $\m^*$ generated by all the homogeneous elements of $\m$; it is a graded prime ideal by \cite[Lemma 1.5.6]{BrunsHerzog}. Then since $\p_1 \subseteq \m^*$ and $\p_i \subseteq \m^*$, and $\m^*$ is properly contained in $\m$, we see that $\p_1 = \m^* = \p_i$. Namely, if $\p_1$, say, were properly contained in $\m^*$, then the Krull dimension of $\Ho (\C) / \p_1$ would be at least $2$, and similarly for $\p_i$. But $\p_1 \neq \p_i$, and so we conclude that $V_1 \cap V_2 = \{ \m_0 \}$. Then by Theorem \ref{thm:fgproperties}(8), the object $M$ admits a direct sum decomposition $M \simeq M_1 \oplus M_2$, where $M_1$ and $M_2$ are objects with $\VM(M_i) = V_i$. Moreover, since $\dim V_i =1$, none of these objects can be zero. This is a contradiction, since the object $M$ is indecomposable, and so $\VM(M)$ must be irreducible. The same proof works for $\VC(X)$.

For the last part of the statement, let $\p$ and $\q$ be the unique minimal primes of $\Supp_{\C}(X)$ and $\Supp_{\M}(M)$, respectively. Then $\VC(X) = Z( \p )$ and $\VM(M) = Z( \q )$, giving
$$\VC(X) \cap \VM(M) = Z( \p ) \cap Z( \q ) = Z( \p + \q )$$
Suppose that $\VC(X) \neq \VM(M)$, so that $\p \neq \q$. Then both $\p$ and $\q$ must be properly contained in $\p + \q$; if $\p = \p + \q$, say, then $\q \subseteq \p$, and the containment would be strict since $\p \neq \q$. But then the dimension of $Z( \q )$ would be greater than that of $Z( \p )$, a contradiction since $\dim Z( \q ) = \dim \VM(M) = 1 = \VC(X) = \dim Z( \p )$ by Theorem \ref{thm:complexityone}. 

Now take any maximal ideal $\m \in \VC(X) \setminus \{ \m_0 \}$. Since $\m$ is not graded, the inclusion $\m^* \subset \m$ is proper, where, as before, $\m^*$ is the graded (prime) ideal of $\Ho (\C)$ generated by all the homogeneous elements of $\m$. As this is necessarily the unique maximal graded ideal contained in $\m$, and $\p$ is graded, we see that $\p$ must equal $\m^*$; otherwise, the dimension of $\VC(X)$ would have been at least $2$, but we know that it is $1$. This implies that $\m$ can not belong to $\VC(X) \cap \VM(M)$, for if it did, then it would have to contain $\p + \q$, which is a graded ideal that strictly contains $\m^*$. 
\end{proof}

In general, if $I$ is an ideal of $\Ho (\C)$, then there is an equality $Z (I) = Z ( \sqrt{I} )$, where $\sqrt{I} $ denotes the radical of $I$. When the finiteness condition \textbf{Fg} holds, then by \cite[Theorem 25]{Matsumura}, the radical of a proper ideal of $\Ho (\C)$ equals the intersection of all the maximal ideals containing it, a fact that we just used in the proof of Proposition \ref{prop:indecperiodic}, and also in Remark \ref{rem:trivialvariety}(3). Consequently, in this setting, whenever $I$ and $J$ are two proper ideals of $\Ho (\C)$, we see that $Z(I) = Z(J)$ if and only if $\sqrt{I} = \sqrt{J}$. We shall use this fact in the proof of the following result, which is the key ingredient in the main theorem; it allows us to reduce the complexities of the objects when we want to establish the module product property for support varieties. A general such reduction result is provided by Theorem \ref{thm:fgproperties}(6), but now we need a much stronger version.

\begin{proposition}\label{prop:reducecomplexity}
Suppose that $\C$ satisfies \emph{\textbf{Fg}}, and that $M \in \M$ is an object with $\cx_{\M}(M) \ge 2$ and $\VM(M)$ irreducible \emph{(}so that $\Supp_{\M} (M)$ contains a unique minimal prime\emph{)}. Then for every $\m \in \VM(M)$ there exists a short exact sequence
$$0 \to W \to \Omega_{\M}^{n}(M) \oplus Q \to M \to 0$$
in $\M$, with the following properties:
\begin{itemize}
\item[(1)] The object $Q$ is projective, and $n \ge 1$;
\item[(2)] $\cx_{\M} (W) = \cx_{\M} (M)-1$;
\item[(3)] $\m \in \VM(W)$.
\end{itemize}
\end{proposition}

\begin{proof}
Let $\p_0$ be the unique minimal (graded) prime of $\Supp_{\M} (M)$, and denote the complexity of $M$ by $d$. Since $\Ho (\C)$ is a finitely generated $k$-algebra by assumption, the quotient $\Ho (\C) / \p_0$ is a finitely generated integral domain. Therefore, by \cite[Corollary 13.4]{Eisenbud}, all the maximal ideals of $\Ho (\C) / \p_0$ are of the same height, namely $\dim \Ho (\C) / \p_0$. The dimension of $\Ho (\C) / \p_0$ equals that of $\Ho (\C) / I_{\M}(M)$, which by definition is the dimension of $\VM(M)$. Thus from Theorem \ref{thm:fgproperties}(1) we see that every maximal ideal of $\Ho (\C) / \p_0$ is of height $d$.

Now let $\m$ be a point in $\VM(M)$. It follows from the above that there exists a strictly increasing chain
$$\p_0 \subset \cdots \subset \p_{d-1} \subset \m$$
in $\Spec \Ho (\C)$. However, by \cite[Theorem 1.5.8]{BrunsHerzog}, there actually exists such a chain in which all the prime ideals $\p_i$ are graded; if $\m \neq \m_0$, we can take as $\p_{d-1}$ the graded ideal $\m^*$ of $\Ho (\C)$, generated by all the homogeneous elements of $\m$ (recall that by \cite[Lemma 1.5.6]{BrunsHerzog}, this is a graded prime ideal).

Take any nonzero homogeneous element $\zeta \in \p_1 \setminus \p_0$; this is possible since $d \ge 2$. Note that since $\p_1 \subseteq \m_0 = \Coh^{+} (\C)$, the degree $n$ of $\zeta$ is positive. From this element we obtain a short exact sequence
$$0 \to L_{\zeta} \to \Omega_{\C}^n( \unit ) \to \unit \to 0$$
in $\C$, where the object $L_{\zeta}$ is Carlson's $L_{\zeta}$-object discussed right before Theorem \ref{thm:fgproperties}. Since $\M$ is an exact $\C$-module category, we obtain a (not necessarily minimal) projective resolution of $M$ when we apply $- \ast M$ to the minimal projective resolution of the unit object $\unit$. Consequently, by Schanuel's Lemma for abelian categories (see \cite[Lemma 2.2]{BPW}), there is an isomorphism $\Omega_{\C}^n( \unit ) \ast M \simeq \Omega_{\M}^n( M ) \oplus Q$ for some projective object $Q \in \M$. As a result, when we  apply $- \ast M$ to the above short exact sequence, we obtain a short exact sequence
$$0 \to L_{\zeta} \ast M \to \Omega_{\M}^n( M ) \oplus Q \to M \to 0$$
in $\M$. By Theorem \ref{thm:fgproperties}(3), there is an equality $\VM( L_{\zeta} \ast M ) = \VM(M) \cap Z( \zeta )$, and so in particular we see that $\m \in \VM( L_{\zeta} \ast M )$.

It remains to show that $\cx_{\M}( L_{\zeta} \ast M ) = d-1$, or, what amounts to the same thing by Theorem \ref{thm:fgproperties}(1), that $\dim \VM( L_{\zeta} \ast M ) = d - 1$. Now
\begin{eqnarray*}
Z \left ( I_{\M}(M) + ( \zeta ) \right ) & = & Z( I_{\M}(M) ) \cap Z( \zeta ) \\
& = & \VM(M) \cap Z( \zeta ) \\
& = & \VM( L_{\zeta} \ast M ) \\
& = & Z( I_{\M}(L_{\zeta} \ast M))
\end{eqnarray*}
hence there is an equality $\sqrt{I_{\M}(L_{\zeta} \ast M)} = \sqrt{I_{\M}(M) + ( \zeta )}$. The dimension of $\VM( L_{\zeta} \ast M )$ is by definition the Krull dimension of $\Ho (\C)/ I_{\M}(L_{\zeta} \ast M)$, which in turn equals that of $\Ho (\C)/ \sqrt{I_{\M}(L_{\zeta} \ast M)}$. Therefore it suffices to show that the Krull dimension of $\Ho (\C)/ \sqrt{I_{\M}(M) + ( \zeta )}$ is $d-1$. For this, consider the chain
$$\p_0 \subset \cdots \subset \p_{d-1} \subset \m$$
of prime ideals from the beginning of the proof. Since $I_{\M}(M)  \subseteq \p_0$ and $\zeta \in \p_1$, the radical ideal $\sqrt{I_{\M}(M) + ( \zeta )}$ is contained in $\p_1$, giving $\dim \Ho (\C)/ \sqrt{I_{\M}(M) + ( \zeta )} \ge d-1$. However, if the inequality were strict, then there would exist a strictly increasing chain
$$\q_0 \subset \cdots \subset \q_d$$
of prime ideals in $\Ho (\C)$, all containing the ideal $\sqrt{I_{\M}(M) + ( \zeta )}$. Since $\zeta \notin \p_0$, and $\p_0$ is the unique minimal prime ideal lying over the ideal $I_{\M}(M)$, we would obtain a strictly increasing chain
$$\p_0 \subset \q_0 \subset \cdots \subset \q_d$$
in $\Supp_{\M}(M)$. But then $\cx_{\M}(M) = \dim \VM(M) \ge d+1$, a contradiction. This shows that the complexity of the object $L_{\zeta} \ast M$ is $d-1$.
\end{proof}

We are now ready to prove the main result.

\begin{theorem}\label{thm:main}
Let $k$ be a field, and $\left ( \C, \ot, \unit \right )$ a finite braided tensor $k$-category satisfying \emph{\textbf{Fg}}. Furthermore, let $\left ( \M, \ast \right )$ be an exact left $\C$-module category, whose set of isomorphism classes of simple objects is finite. Then the following are equivalent:
\begin{itemize}
\item[(1)] $\VM( X \ast M ) = \VC(X) \cap \VM(M)$ for all objects $X \in \C, M \in \M$;
\item[(2)] $\VM( X \ast M ) = \VC(X) \cap \VM(M)$ for all objects $X \in \C, M \in \M$ of complexity $1$;
\item[(3)] $\VM( X \ast M ) = \VC(X) \cap \VM(M)$ for all indecomposable periodic objects $X \in \C, M \in \M$.
\end{itemize}
\end{theorem}

\begin{proof}
If the module product property holds for all objects, then in particular it holds for objects of complexity $1$, hence (2) trivially follows from (1). Since every nonzero indecomposable periodic object has complexity $1$ by Corollary \ref{cor:periodic}, we see that (3) follows from (2). Moreover, (2) follows from (3) because both module products and support varieties respect direct sums. For suppose that $X \in \C$ and $M \in \M$ are objects of complexity $1$, and decompose them both into direct sums $X \simeq X_1 \oplus \cdots \oplus X_s$ and $M \simeq M_1 \oplus \cdots \oplus M_t$, with all the $X_i$ and $M_j$ indecomposable. Then each of these summands is either projective, or periodic by Corollary \ref{cor:periodic}. In general, if $Y \in \C$ and $N \in \M$ are objects, and one of them is projective, then so is $Y \ast N$ by Remark \ref{rem:properties}(2) and the definition of an exact module category, hence both $\VM (Y \ast N)$ and $\VC(Y) \cap \VM(N)$ equal $\{ \m_0 \}$. Therefore, if (3) holds, then
\begin{eqnarray*}
\VM(X \ast M) & = & \VM \left ( \bigoplus_{i,j} \left ( X_i \ast M_j \right ) \right ) \\
& = & \bigcup_{i,j} \VM \left ( X_i \ast M_j \right ) \\
& = & \bigcup_{i,j} \left ( \VC(X_i) \cap \VM(M_j) \right ) \\
& = & \left (  \bigcup_{i} \VC(X_i) \right ) \cap \left (  \bigcup_{j} \VM(M_j) \right ) \\
& = & \VC(X) \cap \VM(M)
\end{eqnarray*}
where we have used Proposition \ref{prop:elementaryproperties}(1). It follows that (2) holds. 

Finally, we will prove that (1) follows from (2). 
Suppose now that (2) holds, and let $X$ and $M$ be arbitrary objects of $\C$ and $\M$, respectively. As above, if one of them is projective, then so is $X \ast M$, and both $\VM(X \ast M)$ and $\VC(X) \cap \VM(M)$ equal $\{ \m_0 \}$. We may therefore suppose that both $X$ and $M$ are nonprojective, that is, that $\cx_{\C}(X) \ge 1$ and $\cx_{\M}(M) \ge 1$. We now argue by induction on the sum $\cx_{\C}(X) + \cx_{\M}(M)$; the assumption being that the module product property holds when this sum is $2$. By Proposition \ref{prop:elementaryproperties}(6), the inclusion $\VM(X \ast M) \subseteq \VC(X) \cap \VM(M)$ holds, hence we must only show the reverse inclusion.

Suppose that $\cx_{\C}(X) + \cx_{\M}(M) > 2$, and that $\cx_{\M}(M) \ge 2$. Let $\p_1, \dots, \p_t$ be the minimal primes of $\Supp_{\M} (M)$, so that $\VM(M) = Z( \p_1 ) \cup \cdots \cup Z( \p_t )$; recall that these primes are graded, by \cite[Lemma 1.5.6]{BrunsHerzog}. We now construct objects $M_1, \dots, M_t \in \M$ with the property that $\VM(M_i) = Z( \p_i )$ and $\VM(X \ast M_i) \subseteq \VM(X \ast M)$ for each $i$. If $t=1$, we simply take $M_1 = M$. If $t \ge 2$, then fix one of the $\p_i$, and let $\zeta_1, \dots, \zeta_s$ be homogeneous elements in $\Ho (\C)$ with $\p_i = ( \zeta_1, \dots, \zeta_s )$. Each $\zeta_j$ gives a short exact sequence
$$0 \to L_{\zeta_j} \to \Omega_{\C}^{n_j}( \unit ) \to \unit \to 0$$
in $\C$, where $n_j$ is the degree of $\zeta_j$. Now take $M_i = L_{\zeta_s} \ast \cdots \ast L_{\zeta_1} \ast M$. Then from Theorem \ref{thm:fgproperties}(3) we obtain
\begin{eqnarray*}
\VM(M_i) & = & \VM(M) \cap Z( \zeta_1 ) \cap \cdots \cap Z( \zeta_s ) \\
& = & \VM(M) \cap Z \left ( ( \zeta_1, \dots, \zeta_s ) \right ) \\
& = &  \VM(M) \cap Z( \p_i ) \\
& = & Z( \p_i )
\end{eqnarray*}
Next, denote the object $L_{\zeta_j} \ast \cdots \ast L_{\zeta_1} \ast M$ by $N_j$ for $1 \le j \le s$, and put $N_0 = M$. By applying $- \ast N_{j-1}$ to the exact sequence above, we obtain an exact sequence
$$0 \to N_j \to \Omega_{\M}^{n_j}( N_{j-1} ) \oplus Q_j \to N_{j-1} \to 0$$
in $\M$ (for some projective object $Q_j$), on which we apply $X \ast -$ and obtain an exact sequence
$$0 \to X \ast N_j \to \Omega_{\M}^{n_j}( X \ast N_{j-1} ) \oplus P_j \to X \ast N_{j-1} \to 0$$
for some projective object $P_j$. To get the middle terms in these two sequences, we have applied Schanuel's Lemma for abelian categories (see \cite[Lemma 2.2]{BPW}), together with the fact that the module product commutes with syzygies up to projective objects. From the properties listed in Proposition \ref{prop:elementaryproperties}, we now obtain the inclusions
$$\VM(X \ast M_i) = \VM(X \ast N_s) \subseteq \cdots \subseteq \VM(X \ast N_0) = \VM(X \ast M)$$
hence the object $M_i$ has the properties that we wanted. We now claim that if we can show that the inclusion $\VC(X) \cap \VM(M_i) \subseteq \VM(X \ast M_i)$ holds for each $i$, then we are done. Namely, if this is the case, then
\begin{eqnarray*}
\VC(X) \cap \VM(M) & = & \VC(X) \cap \left ( \bigcup_{i=1}^t Z( \p_i ) \right )  \\
& = & \VC(X) \cap \left ( \bigcup_{i=1}^t \VM(M_i) \right )  \\
& = & \bigcup_{i=1}^t \left ( \VC(X) \cap \VM(M_i) \right ) \\
& \subseteq &  \bigcup_{i=1}^t \VM(X \ast M_i) \\
& \subseteq & \VM(X \ast M)
\end{eqnarray*}
This proves the claim.

What remains to show is that the inclusion $\VC(X) \cap \VM(M_i) \subseteq \VM(X \ast M_i)$ holds for each $i$. To do this, note first that $\cx_{\M}(M_i) \le \cx_{\M}(M)$. Namely, the primes $\p_1, \dots, \p_t$ are the minimal ones in $\Supp_{\M}(M)$, whereas $\p_i$ is the only minimal prime in $\Supp_{\M}(M_i)$. Thus $\dim \VM(M)$ is the length of the longest chain in $\Spec \Ho (\C)$ starting with one of the primes $\p_1, \dots, \p_t$, and $\dim \VM(M_i)$ is the length of the longest chain in $\Spec \Ho (\C)$ starting with $\p_i$. Consequently, from Theorem \ref{thm:fgproperties}(1) we see that $\cx_{\M}(M_i) \le \cx_{\M}(M)$.

If $\cx_{\M}(M_i) \le \cx_{\M}(M) -1$, then $\cx_{\C}(X) + \cx_{\M}(M_i) \le \cx_{\C}(X) + \cx_{\M}(M)-1$, and so by induction $\VC(X) \cap \VM(M_i) \subseteq \VM(X \ast M_i)$ holds in this case. If on the other hand $\cx_{\M}(M_i) = \cx_{\M}(M)$, then since $\cx_{\M}(M_i) \ge 2$ and $\Supp_{\M}(M_i)$ contains a unique minimal prime, we can apply Proposition \ref{prop:reducecomplexity}; for each $\m \in \VM(M_i)$ there exists a short exact sequence
$$0 \to W( \m ) \to \Omega_{\M}^{n( \m )}(M_i) \oplus Q( \m ) \to M_i \to 0$$
in which the object $Q( \m )$ is projective, the complexity of the object $W( \m )$ is $\cx_{\M} (M) - 1$, and $\m \in \VM( W( \m ))$. Note that $\VM( W( \m ) ) \subseteq \VM(M_i)$ by Proposition \ref{prop:elementaryproperties}, and consequently that 
$$\bigcup_{\m \in \VM(M_i)} \VM( W( \m ) ) = \VM(M_i)$$ 
since $\m \in \VM( W( \m ))$.

As explained earlier in this proof, when we apply $X \ast -$ to the sequence we just obtained, the result is a short exact sequence
$$0 \to X \ast W( \m ) \to \Omega_{\M}^{n( \m )}(X \ast M_i) \oplus P( \m ) \to X \ast M_i \to 0$$
where the object $P( \m )$ is projective. Using Proposition \ref{prop:elementaryproperties} again, we obtain the inclusion $\VM( X \ast W( \m ) ) \subseteq \VM( X \ast M_i)$, and by induction we also see that $\VC(X) \cap \VM(W( \m )) \subseteq \VM(X \ast W( \m ))$, since $\cx_{\C}(X) + \cx_{\M}(W( \m )) = \cx_{\C}(X) + \cx_{\M}(M)-1$. Combining everything, we now obtain
\begin{eqnarray*}
\VC(X) \cap \VM(M_i) & = & \VC(X) \cap \left ( \bigcup_{\m \in \VM(M_i)} \VM( W( \m ) ) \right ) \\
& = & \bigcup_{\m \in \VM(M_i)} \left ( \VC(X) \cap \VM( W( \m ) ) \right ) \\
& \subseteq & \bigcup_{\m \in \VM(M_i)} \VM(X \ast W( \m )) \\
& \subseteq & \VM \left ( X \otimes M_i \right )
\end{eqnarray*}
This concludes the induction proof in the case when $\cx_{\M}(M) \ge 2$. 

Finally, if $\cx_{\M}(M) = 1$ and $\cx_{\C}(X) \ge 2$, then we use virtually the same arguments to reach the conclusion. Namely, we reduce the complexity of $X$ while keeping the object $M$ fixed. We have shown that (1) follows from (2). 
\end{proof}

Thus in order to verify that the module product property holds for support varieties, it is enough to check that it holds for the indecomposable periodic objects. The following result provides an alternative way of verifying all this, by considering whether certain module products are projective or not.

\begin{theorem}\label{thm:mainalternative}
Let $k$ be a field, and $\left ( \C, \ot, \unit \right )$ a finite braided tensor $k$-category satisfying \emph{\textbf{Fg}}. Furthermore, let $\left ( \M, \ast \right )$ be an exact left $\C$-module category, whose set of isomorphism classes of simple objects is finite. Then the following are equivalent:
\begin{itemize}
\item[(1)] $\VM( X \ast M ) = \VC(X) \cap \VM(M)$ for all objects $X \in \C, M \in \M$;
\item[(2)] For all objects $X \in \C, M \in \M$, if $\VC(X) \cap \VM(M) \neq \{ \m_0 \}$, then $X \ast M$ is not projective;
\item[(3)] For all objects $X \in \C, M \in \M$ of complexity $1$, if $\VC(X) \cap \VM(M) \neq \{ \m_0 \}$, then $X \ast M$ is not projective;
\item[(4)] For all indecomposable periodic objects $X \in \C, M \in \M$, if $\VC(X) \cap \VM(M) \neq \{ \m_0 \}$, then $X \ast M$ is not projective;
\item[(5)] For all nonzero indecomposable periodic objects $X \in \C, M \in \M$ with $\VC(X) = \VM(M)$, the object $X \ast M$ is not projective.
\end{itemize}
\end{theorem}

\begin{proof}
Suppose that (1) holds, and let $X \in \C, M \in \M$ be objects with $\VC(X) \cap \VM(M) \neq \{ \m_0 \}$. Then $\VM( X \ast M ) \neq \{ \m_0 \}$, hence $X \ast M$ cannot be projective. This shows that (1) implies (2). The implication (2) $\Rightarrow$ (3) is trivial, the implication (3) $\Rightarrow$ (4) follows from Corollary \ref{cor:periodic}, and the implication (4) $\Rightarrow$ (5) follows from the fact that the support variety of a nonzero periodic object is non-trivial by Theorem \ref{thm:complexityone}.

Finally, suppose that (5) holds. By Theorem \ref{thm:main}, in order to show that (1) holds, it is enough to show that $\VM( X \ast M ) = \VC(X) \cap \VM(M)$ for all indecomposable periodic objects $X \in \C, M \in \M$. This is trivially true if either $X$ or $M$ is zero, so suppose that they are both nonzero, indecomposable and periodic. If $\VC(X) \cap \VM(M) = \{ \m_0 \}$, then $\VM ( X \ast M ) = \{ \m_0 \}$ by Proposition \ref{prop:elementaryproperties}(6), and we are done. If $\VC(X) \cap \VM(M) \neq \{ \m_0 \}$, then $\VC(X) = \VM(M)$ by Proposition \ref{prop:indecperiodic}, and so by assumption the object $X \ast M$ is not projective. Then $\dim \VM( X \ast M ) \ge 1$ by (1) and (2) of Theorem \ref{thm:fgproperties}, that is, the Krull dimension of $\Ho (\C) /  I_{\M}(X \ast M)$, and therefore also of $\Ho (\C) / \sqrt{I_{\M}(X \ast M)}$, is at least $1$. Now apply Proposition \ref{prop:indecperiodic} once more; let $\p$ be the unique minimal prime of $\Supp_{\C}(X)$, so that $\VC(X) = Z( \p ) = \VM(M)$. Then 
$$Z \left ( I_{\M}(X \ast M) \right ) = \VM(X \ast M) \subseteq \VC(X) = Z( \p )$$
by  Proposition \ref{prop:elementaryproperties}(6), giving $\p \subseteq \sqrt{I_{\M}(X \ast M)}$ by \cite[Theorem 25]{Matsumura}. If this inequality is strict, then $\Supp_{\C}(X)$ contains a strictly increasing chain of length at least $2$, since $\dim \Ho (\C) / \sqrt{I_{\M}(X \ast M)} \ge 1$. This is impossible since $\dim \VC(X) =1$ by Theorem \ref{thm:complexityone}, hence $\p = \sqrt{I_{\M}(X \ast M)}$. But then
$$\VM(X \ast M) = Z \left ( \sqrt{I_{\M}(X \ast M)} \right ) = Z( \p ) = \VC(X) \cap \VM(M)$$
and we are done.
\end{proof}

\section{Skew group algebras and symmetric tensor categories}\label{sec:symmetric}

In this section, we apply the results of Section \ref{sec:tensor},
specifically Theorem~\ref{thm:mainalternative}(5), to 
categories of finite dimensional representations of certain 
skew group algebras. 
For these tensor categories, the finiteness condition \textbf{Fg} holds, 
and we shall see that the tensor product property holds for support varieties. 
Using Deligne's classification theorem from \cite{Deligne2}, we obtain 
as a special case an important class of examples, 
namely the finite symmetric tensor categories over 
algebraically closed ground fields of characteristic zero, 
giving a new proof of the tensor product property for these categories
(cf.~Drupieski and Kujawa~\cite[Corollary 3.2.4]{DK}).
In case the group has order two and the characteristic of $k$ is odd,
our result should also be compared with Benson, Iyengar, Krause, and 
Pevtsova~\cite[Theorems 8.10 and 9.3]{BIKP}.

The skew group algebras in which we are interested arise from group actions on exterior algebras, so let us fix some notation that we will use throughout this section. Let $k$ be a field, $c$ a positive integer, and $\Lambda$ the exterior algebra in $c$ indeterminates $x_1, \dots, x_c$ over $k$:
$$\Lambda = k \langle x_1, \dots, x_c \rangle / (x_i^2, x_ix_j + x_jx_i)$$
Furthermore, let $G$ be a finite group acting on $\Lambda$, via a homomorphism into the group of algebra automorphisms of $\Lambda$. We may then form the skew group algebra $\Lambda \rtimes G$. As a $k$-vector space, this is just $\Lambda \otimes_k kG$, which is finite dimensional, and every element is of the form $\sum_{g \in G} w_g \otimes g$ for some $w_g \in \Lambda$. Multiplication is defined by
$$(w_1 \otimes g_1)(w_2 \otimes g_2) = w_1(^{g_1}w_2) \otimes g_1g_2$$ 
for $w_i \in \Lambda$ and $g_i \in G$. The skew group algebra is often also called the smash product algebra, and then typically denoted by $\Lambda \# kG$. If the characteristic of $k$ does not divide the order of $G$, then since exterior algebras are selfinjective, it follows from \cite[Theorem 1.1 and Theorem 1.3]{ReitenRiedtmann} that $\Lambda \rtimes G$ is also selfinjective. Finally, note that the natural inclusion $\Lambda \to \Lambda \rtimes G$ given by $w \mapsto w \ot e$ (where $e$ is the identity element of $G$) turns $\Lambda \rtimes G$ into a left and right $\Lambda$-module, in both cases free of rank $| G |$.

\begin{remark}\label{rem:fg}
(1) Suppose that $\Lambda \rtimes G$ happens to be a Hopf algebra, and that the characteristic of $k$ does not divide the order of $G$. Then the finite tensor category $\mod ( \Lambda \rtimes G )$ of finitely generated left modules satisfies \textbf{Fg}. To see this, denote the algebra by $H$. By \cite[Theorem 4.1(2)]{Bergh}, the Hochschild cohomology ring $\Hoch^*( H )$ is Noetherian, and for every $H$-bimodule $X$, the right $\Hoch^*(H)$-module $\Ext_{H^e}^*(H,X)$ is finitely generated (here $H^e$ denotes the enveloping algebra $H \otimes_k H^{\op}$). By \cite[Lemma 3.2]{Bergh}, this implies that $\Ext_H^*(M,M)$ is a finitely generated $\Hoch^*(H)$-module, for every finitely generated left $H$-module $M$. Finally, by \cite[Lemma 4.2]{Bergh}, this in turn implies that the finite tensor category $\mod H$ satisfies \textbf{Fg}.

(2) Given any ring $R$ together with an automorphism $\psi \colon R \to R$, we may twist a left module $X$ and obtain a module ${_{\psi}X}$. The underlying abelian group is the same, but the module action becomes $r \cdot x = \psi (r) x$ for $r \in R$ and $x \in X$. There is an isomorphism ${_{\psi}X} \simeq {_{\psi}R} \ot_R X$, hence twisting induces an exact functor. In particular, for $\Lambda$ and $G$, every $g \in G$ acts on the cohomology ring $\Ext_{\Lambda}^*(k,k)$ by twisting of extensions. That is, given a homogeneous element $\eta$ realized as an extension
$$0 \to k \xrightarrow{f_0} X_1 \xrightarrow{f_1} \cdots \xrightarrow{f_{n-1}} X_n \xrightarrow{f_n} k \to 0$$
we obtain the element $^g\eta$ realized as the extension
$$0 \to k \xrightarrow{f_0} {_gX_1} \xrightarrow{f_1} \cdots \xrightarrow{f_{n-1}} {_gX_n} \xrightarrow{f_n} k \to 0$$
Here we have used the notation ${_gX}$ for the $\Lambda$-module obtained from $X$ by twisting with the automorphism on $\Lambda$ given by $g$; note that ${_gk} = k$.

(3) Suppose, as in (1), that $H = \Lambda \rtimes G$ is a Hopf algebra, and that the characteristic of $k$ does not divide the order of $G$. Then the cohomology ring $\Ext_H^*(k,k)$ is isomorphic to the $G$-invariant subring $\Ext_{\Lambda}^*(k,k)^{G}$ of $\Ext_{\Lambda}^*(k,k)$, via the restriction map
$$\Ext_H^{*}(k,k) \xrightarrow{\tau_{H, \Lambda}^*(k,k)} \Ext_{\Lambda}^{*}(k,k)$$
see, for example, \cite[Theorem 2.17]{StefanVay}.
\end{remark}

The following lemma shows that if we take any subalgebra of 
$\Lambda \rtimes G$ containing the exterior algebra $\Lambda$, 
then restriction of cohomology is injective.

\begin{lemma}\label{lem:restriction}
If the characteristic of $k$ does not divide the order of $G$, then for any algebra $A$ with $\Lambda \subseteq A \subseteq \Lambda \rtimes G$, and any pair of $\Lambda \rtimes G$-modules $M,N$, the restriction map
$$\Ext_{\Lambda \rtimes G}^{*}(M,N) \xrightarrow{\tau_{\Lambda \rtimes G,A}^*(M,N)} \Ext_{A}^{*}(M,N)$$
is injective.
\end{lemma}

\begin{proof}
Let us denote $\Lambda \rtimes G$ by $H$. The composition of restriction maps
$$\Ext_{H}^{*}(M,N) \xrightarrow{\tau_{H,A}^*(M,N)} \Ext_{A}^{*}(M,N) \xrightarrow{\tau_{A, \Lambda}^*(M,N)} \Ext_{\Lambda}^{*}(M,N)$$
equals the restriction map $\tau_{H, \Lambda}^*(M,N)$ from $H$ to $\Lambda$. It therefore suffices to show that the latter is injective. If $\theta \in \Ext_{H}^{n}(M,N)$ for some $n$, then its restriction to $\Ext_{\Lambda}^{n}(M,N)$ is $H \ot_{H} \theta$, where we view $H$ as a $\Lambda$-$H$-bimodule. Inducing back to $H$, we obtain the element $H \ot_{\Lambda} H \ot_H \theta \in \Ext_{H}^{n}( H \ot_{\Lambda} H \ot_H M, H \ot_{\Lambda} H \ot_H N)$, where we view the leftmost $H$ in the tensor products as an $H$-$\Lambda$-bimodule. By \cite[Theorem 1.1]{ReitenRiedtmann}, $H$ is a direct summand of $H \ot_{\Lambda} H$ as a bimodule over itself, and so we see that $\theta$ is a direct summand of $H \ot_{\Lambda} H \ot_H \theta$. This shows that the restriction map from $\Ext_{H}^{*}(M,N)$ to $\Ext_{\Lambda}^{*}(M,N)$ is injective.
\end{proof}

Suppose now that the characteristic of $k$ is not $2$, and let $C_2$ be a (multiplicative) group of order $2$, say $C_2 = \{ e,h \}$ with $h^2 = e$. This group acts on $\Lambda$ by defining $^hx_i = -x_i$ for each $i$. From now on, we set $$A = \Lambda \rtimes C_2.$$ As a $k$-algebra, it is isomorphic to the algebra generated by $h,x_1, \dots, x_c$, with relations $h^2=1, x_i^2 = 0, x_ix_j+x_jx_i=0$ and $hx_i+x_ih=0$. We see that it is a Hopf algebra by defining a comultiplication $\Delta$, antipode $S$ and counit $\varepsilon$ as follows: $\Delta (h) = h \otimes h, \Delta (x_i) = x_i \otimes 1 + h \otimes x_i, S(h)=h, S(x_i) = -hx_i, \varepsilon (h) =1$ and $\varepsilon (x_i) =0$. 

The finite tensor category $\mod A$ of finitely generated left $A$-modules is symmetric. To see this, take two modules $M,N \in \mod A$, and decompose them into subspaces
$$M = M_0 \oplus M_1, \hspace{5mm} N = N_0 \oplus N_1$$
given by eigenspaces for the action of $h$; this is possible since the characteristic of $k$ is not $2$. Thus $hm_0 = m_0$ and $hm_1 = -m_1$ whenever $m_i \in M_i$, and similarly for $N$. One now checks that the map $M \ot N \to N \ot M$ given by
$$m_i \ot n_j \mapsto (-1)^{ij} n_j \ot m_i$$
is a functorial isomorphism, and it squares to the identity. Hence $\mod A$ is symmetric. Moreover, by Remark \ref{rem:fg}(1), it also satisfies \textbf{Fg}. For a module $M \in \mod A$, we shall denote the support variety $V_{\mod A}(M)$ by just $V_A(M)$; these are defined in terms of the maximal ideal spectrum of the (commutative) even degree cohomology ring $\Ext_A^{2*}(k,k)$. We denote by $\m_0$ the unique graded maximal ideal of this ring.

\begin{remark}\label{rem:evendegrees}
By Remark \ref{rem:fg}(3), the cohomology ring $\Ext_A^*(k,k)$ is isomorphic to $\Ext_{\Lambda}^*(k,k)^{C_2}$ via the restriction map
$$\Ext_A^{*}(k,k) \xrightarrow{\tau_{A, \Lambda}^*(k,k)} \Ext_{\Lambda}^{*}(k,k)$$
The action of $C_2$ on $\Ext_{\Lambda}^{*}(k,k)$ is quite simple: the generator $h \in C_2$ acts as $(-1)^n$ on $\Ext_{\Lambda}^n(k,k)$. This can be seen from the action of $h$ on the Koszul resolution of $k$ in degree $n$, induced by the action of $h$ on each $x_i$ as multiplication by $-1$. Thus $\Ext_{\Lambda}^*(k,k)^{C_2}$ is nothing but the even degree subspace $\Ext_{\Lambda}^{2*}(k,k)$ of $\Ext_{\Lambda}^*(k,k)$. In particular, we see that $\Ext_A^n(k,k) = 0$ for odd $n$, so that $\Ext_A^*(k,k) = \Ext_A^{2*}(k,k)$.
\end{remark}

Now take a $c$-tuple $\lambda = ( \lambda_1, \dots, \lambda_c ) \in k^c$, and denote the element $\lambda_1x_1 + \cdots + \lambda_cx_c$ of $\Lambda$ by $u_{\lambda}$. Then $u_{\lambda}^2 = 0$, and so the subalgebra $k[ u_{\lambda} ]$ generated by $u_{\lambda}$ is isomorphic to the truncated polynomial ring $k[y]/(y^2)$ whenever $\lambda$ is nonzero. For every such $c$-tuple $\lambda$, the algebra $\Lambda$ is free as a left and as a right module over the subalgebra $k[ u_{\lambda} ]$; this follows, for example, from \cite[Theorem 2.6]{BensonErdmannHolloway}. Combining with the above, we see that the same holds for the algebra $A$. 

\sloppy The inclusion $k[ u_{\lambda} ] \to A$ gives a restriction map
$$\Ext_A^{*}(k,k) \xrightarrow{\tau_{A, \lambda}^*(k,k)} \Ext_{k[ u_{\lambda} ]}^{*}(k,k)$$
We denote by $\tau_{A, \lambda}^{2*}(k,k)$ the restriction of this map to the even cohomology ring $\Ext_A^{2*}(k,k)$. Of course, since $\Ext_A^*(k,k) = \Ext_A^{2*}(k,k)$, we have not in practice restricted the map $\tau_{A, \lambda}^{*}(k,k)$ to a subalgebra. The first result we prove is that when $\lambda$ is a nonzero $c$-tuple, then the kernel of this map is a graded prime ideal of $\Ext_A^{2*}(k,k)$. Moreover, two $c$-tuples give rise to different prime ideals if and only if they are not on the same line.

\begin{lemma}\label{lem:kernel}
For every nonzero $c$-tuple $\lambda \in k^c$, the ideal $\Ker \tau_{A,\lambda}^{2*}(k,k)$ is a graded prime ideal of $\Ext_A^{2*}(k,k)$, different from $\m_0$. Moreover, if $\mu$ is another nonzero $c$-tuple, then $\Ker \tau_{A,\lambda}^{2*}(k,k) = \Ker \tau_{A,\mu}^{2*}(k,k)$ if and only if $\mu = \alpha \lambda$ for some \emph{(}necessarily nonzero\emph{)} scalar $\alpha \in k$.
\end{lemma}

\begin{proof}
Let $\lambda$ be a nonzero $c$-tuple in $k^c$. Since $k[ u_{\lambda} ]$ is isomorphic to the truncated polynomial ring $k[y]/(y^2)$, the cohomology ring $\Ext_{k[ u_{\lambda} ]}^{*}(k,k)$ is isomorphic to a polynomial ring $k[z]$ with $z$ in degree one. In particular, the even cohomology ring $\Ext_{k[ u_{\lambda} ]}^{2*}(k,k)$ is an integral domain. It follows that if $\eta$ and $\theta$ are elements of $\Ext_A^{2*}(k,k)$ with $\eta \theta \in \Ker \tau_{A,\lambda}^{2*}(k,k)$, then either $\eta \in \Ker \tau_{A,\lambda}^{2*}(k,k)$ or $\theta \in \Ker \tau_{A,\lambda}^{2*}(k,k)$, since the restriction map is a ring homomorphism. Thus $\Ker \tau_{A,\lambda}^{2*}(k,k)$ is a prime ideal since it is proper (it does not contain the identity element $1 \in \Hom_A(k,k)$, for example). 

Now take another nonzero $c$-tuple $\mu \in k^c$. If $\mu = \alpha \lambda$ for some $\alpha \in k$, then $u_{\mu} = \alpha u_{\lambda}$, and so $k[ u_{\mu} ] = k[ u_{\lambda} ]$ as subalgebras of $A$. Then trivially $\Ker \tau_{A,\lambda}^{2*}(k,k) = \Ker \tau_{A,\mu}^{2*}(k,k)$. Note that when $c=1$, then $\mu$ must be on the same line as $\lambda$.

Conversely, suppose that $\lambda$ and $\mu$ are not on the same line (so $c$ must be at least $2$), and consider the linear map $\phi_{\lambda} \colon k^c \to k$ given by $\rho \mapsto \langle \lambda, \rho \rangle$, where $\langle \lambda, \rho \rangle = \lambda_1\rho_1 + \cdots + \lambda_c \rho_c$. This map is surjective since $\lambda$ is nonzero, and so $\Ker \phi_{\lambda}$ is of dimension $c-1$. Now choose a basis for $\Ker \phi_{\lambda}$, and consider the $(c-1) \times c$-matrix $E$ whose rows are these $c$-tuples, in any order. The rank of $E$ is $c-1$, and so its null space is of dimension one, and contains $\lambda$. Since $\mu$ is not on the same line as $\lambda$, it cannot belong to the nullspace, i.e.\ $E \mu \neq 0$. Consequently, there exists a $c$-tuple $\rho \in k^c$ with $\langle \lambda, \rho \rangle = 0$ and $\langle \mu, \rho \rangle \neq 0$ (for example, one of the rows of $E$ has this property). Choose one such $c$-tuple $\rho$.

Consider the projective cover
$$0 \to I \to \Lambda \to k \to 0$$
of $k$ as a left $\Lambda$-module, where $I$ is the left ideal $(x_1, \dots, x_c) \subseteq \Lambda$. Furthermore, look at the map $I \to k$ given by
$$\beta_1x_1 + \cdots + \beta_cx_c + w \mapsto \langle \beta, \rho \rangle$$
for $w \in I^2$ and $\beta = ( \beta_1, \dots, \beta_c )$. This map is a $\Lambda$-homomorphism mapping $u_{\lambda}$ to zero and $u_{\mu}$ to something nonzero, and does not factor through $\Lambda$. Consequently, it represents a nonzero element $\eta \in \Ext_{\Lambda}^1(k,k)$. Now for any nonzero $c$-tuple $\sigma \in k^c$, the ideal $I$ decomposes over $k[ u_{\sigma} ]$ as $( u_{\sigma} ) \oplus Q_{\sigma} $, for some free $k[ u_{\sigma} ]$-module $Q_{\sigma} $. Furthermore, the restriction map
$$\Ext_{\Lambda}^{*}(k,k) \xrightarrow{\tau_{\Lambda, \sigma}^{*}(k,k)} \Ext_{k[ u_{\sigma} ]}^{*}(k,k)$$
maps $\eta$ to the element of $\Ext_{k[ u_{\sigma} ]}^{1}(k,k)$ represented by the map $( u_{\sigma} ) \oplus Q_{\sigma} \to k$ given by $\alpha u_{\sigma} + q \mapsto \alpha \langle \sigma, \rho \rangle$ for $\alpha \in k$ and $q \in Q_{\sigma}$. Then $\tau_{\Lambda,\lambda}^{*}(k,k)( \eta ) =0$, whereas $\tau_{\Lambda,\mu}^{*}(k,k)( \eta ) \neq 0$ since $\langle \mu, \rho \rangle \neq 0$, and the $k[ u_{\mu} ]$-homomorphism $( u_{\mu} ) \oplus Q_{\mu} \to k$ above does not factor through $\Lambda$. The restriction maps are ring homomorphisms, hence $\tau_{\Lambda,\lambda}^{*}(k,k)( \eta^2 ) =0$ and $\tau_{\Lambda,\mu}^{*}(k,k)( \eta^2 ) \neq 0$, the latter because $\Ext_{k[ u_{\mu} ]}^{*}(k,k)$ is an integral domain.

For every nonzero $c$-tuple $\sigma \in k^c$, the inclusions $k[ u_{\sigma} ] \to \Lambda \to A$ of $k$-algebras induce the sequence
$$\Ext_A^{*}(k,k) \xrightarrow{\tau_{A, \Lambda}^*(k,k)} \Ext_{\Lambda}^*(k,k) \xrightarrow{\tau_{\Lambda, \sigma}^*(k,k)} \Ext_{k[ u_{\sigma} ]}^{*}(k,k)$$
of restriction maps. The composition equals the restriction map $\tau_{A, \sigma}^*(k,k)$. Now by Remark \ref{rem:evendegrees}, the restriction map
$$\Ext_A^{2*}(k,k) \xrightarrow{\tau_{A, \Lambda}^{2*}(k,k)} \Ext_{\Lambda}^{2*}(k,k)$$
is an isomorphism, and so the element $\eta^2 \in \Ext_{\Lambda}^2(k,k)$ belongs to the image of $\tau_{A, \Lambda}^*(k,k)$, where $\eta \in \Ext_{\Lambda}^1(k,k)$ is the element from above. Choosing an element $\theta \in \Ext_A^2(k,k)$ such that $\tau_{A, \Lambda}^*(k,k)( \theta ) = \eta^2$, we obtain
$$\tau_{A, \sigma}^{2*}(k,k)( \theta ) = \tau_{A, \sigma}^*(k,k)( \theta ) = \tau_{\Lambda, \sigma}^*(k,k) \left ( \tau_{A, \Lambda}^*(k,k)( \theta ) \right ) = \tau_{\Lambda, \sigma}^*(k,k) ( \eta^2 )$$
for every nonzero $c$-tuple $\sigma \in k^c$. We showed above that $\tau_{\Lambda,\lambda}^{*}(k,k)( \eta^2 ) =0$ whereas $\tau_{\Lambda,\mu}^{*}(k,k)( \eta^2 ) \neq 0$, and so $\theta$ is an element of $\Ker \tau_{A, \lambda}^{2*}(k,k)$, but not of $\Ker \tau_{A, \mu}^{2*}(k,k)$. This shows that $\Ker \tau_{A, \lambda}^{2*}(k,k)$ does not equal $\Ker \tau_{A, \mu}^{2*}(k,k)$ when $\lambda$ and $\mu$ are not on the same line.

Finally, we prove that $\Ker \tau_{A, \lambda}^{2*}(k,k)$ does not equal the graded maximal ideal $\m_0$ of $\Ext_A^{2*}(k,k)$. If $c=1$, then $u_{\lambda}$ is just the generator $x_1$ multiplied with a nonzero scalar, and so $k[ u_{\lambda} ] = \Lambda$ in this case. The restriction map from $\Ext_A^{2*}(k,k)$ to $\Ext_{\Lambda}^{2*}(k,k)$ is an isomorphism, hence $\Ker \tau_{A, \lambda}^{2*}(k,k) = 0 \neq \m_0$. When $c \ge 2$, we proved above that for the $c$-tuple $\mu$ the element $\eta^2 \in \Ext_A^2(k,k)$ did not belong to $\Ker \tau_{A, \mu}^{2*}(k,k)$. Thus $\Ker \tau_{A, \mu}^{2*}(k,k) \neq \m_0$, and by switching the roles of $\lambda$ and $\mu$ we see that also $\Ker \tau_{A, \lambda}^{2*}(k,k) \neq \m_0$.
\end{proof}

We now turn our attention to a class of $A$-modules whose support varieties are determined by the prime ideals of $\Ext_A^{2*}(k,k)$ considered in the lemma. Namely, for a nonzero $c$-tuple $\lambda \in k^c$, denote the left $A$-module $A( u_{\lambda} \ot e )$ by just $Au_{\lambda}$. Analogues of these modules have been used earlier, in particular in connection with rank varieties; see \cite{BensonErdmannHolloway, BerghErdmann, PevtsovaWitherspoon}. In the following result, we establish the properties that we need for $Au_{\lambda}$; see also \cite[Section 2]{PevtsovaWitherspoon}. Recall that $\stHom$ denotes the quotient of the space of homomorphisms by the subspace of those factoring through a projective module. 

\begin{proposition}\label{prop:periodictestmodule}
For every nonzero $c$-tuple $\lambda \in k^c$, the following hold.

\emph{(1)} The $A$-module $Au_{\lambda}$ is $1$-periodic, i.e.\ $\Omega_A^1 ( Au_{\lambda} ) \simeq Au_{\lambda}$. Moreover, it is isomorphic to the induced module $A \ot_{k[u_{\lambda}]} k$.

\emph{(2)} A module $M \in \mod A$ is free as a $k[u_{\lambda}]$-module if and only if $\stHom_A( Au_{\lambda},M )=0$.

\emph{(3)} $\Ext_A^n(Au_{\lambda},k) \neq 0$ for every positive integer $n$, and the restriction map
$$\Ext_A^* \left ( Au_{\lambda},k \right ) \xrightarrow{\tau_{A, \lambda}^*(Au_{\lambda},k)} \Ext_{k[u_{\lambda}]}^* \left ( Au_{\lambda},k \right )$$
is injective in every positive degree.

\emph{(4)} $V_A ( Au_{\lambda} ) = Z( \Ker \tau_{A,\lambda}^{2*}(k,k) )$, and this variety is irreducible.
\end{proposition}

\begin{proof}
By \cite[Lemma 2.14]{BensonErdmannHolloway}, the sequence
$$\cdots \to \Lambda \xrightarrow{\cdot u_{\lambda}} \Lambda \xrightarrow{\cdot u_{\lambda}} \Lambda \xrightarrow{\cdot u_{\lambda}} \Lambda \to \cdots$$
of left $\Lambda$-modules is exact. Applying $A \ot_{\Lambda} -$, we obtain an exact sequence of left $A$-modules, since $A$ is free as a right $\Lambda$-module. The canonical isomorphism $A \ot_{\Lambda} \Lambda \simeq A$ then gives an exact sequence
\begin{equation*} \label{eqn:sequence}
\cdots \to A \xrightarrow{\cdot (u_{\lambda} \ot e)} A \xrightarrow{\cdot (u_{\lambda} \ot e)} A \xrightarrow{\cdot (u_{\lambda} \ot e)} A \to \cdots \tag{$\dagger$}
\end{equation*}
of left $A$-modules, hence $Au_{\lambda}$ is $1$-periodic. The last part of (1) follows from the isomorphisms
$$A \ot_{k[u_{\lambda}]} k \simeq A \ot_{k[u_{\lambda}]} k[ u_{\lambda} ]/( u_{\lambda} ) \simeq A / Au_{\lambda} = A / A(u_{\lambda} \ot e)$$
of left $A$-modules, together with the isomorphism $A / A(u_{\lambda} \ot e) \simeq A(u_{\lambda} \ot e)$ which is immediate from the exact sequence (\ref{eqn:sequence}).

For (2), we use the isomorphism from (1) together with the  Eckmann-Shapiro Lemma, and obtain
$$\stHom_A \left ( Au_{\lambda}, M \right ) \simeq \stHom_A \left ( A \ot_{k[u_{\lambda}]} k, M \right ) \simeq \stHom_{k[u_{\lambda}]} \left ( k,M \right )$$
Since the algebra $k[u_{\lambda}]$ is isomorphic to $k[y]/(y^2)$, the $k[u_{\lambda}]$-module $M$ is free if and only if it does not contain $k$ as a direct summand. Consequently, it is free if and only if $\stHom_{k[u_{\lambda}]} ( k,M ) =0$.

For (3), we use the periodicity of $Au_{\lambda}$ and the fact that $A$ is selfinjective to obtain
$$\Ext_A^n \left ( Au_{\lambda},k \right ) \simeq \stHom_A \left ( \Omega_A^n(Au_{\lambda}),k \right ) \simeq \stHom_A \left ( Au_{\lambda},k \right )$$
for every positive integer $n$. From (2) we see that $\stHom_A ( Au_{\lambda},k ) \neq 0$, and so $\Ext_A^n ( Au_{\lambda},k ) \neq 0$ as well.

For the restriction map, note first that since $A$ is free as a left $k[u_{\lambda}]$-module, the sequence (\ref{eqn:sequence}) restricts to a sequence of free $k[u_{\lambda}]$-modules. Therefore $\Omega_{k[u_{\lambda}]}^n ( Au_{\lambda} )$ is stably isomorphic to  $Au_{\lambda}$ for every $n \ge 1$, giving
$$\Ext_{k[u_{\lambda}]}^n \left ( Au_{\lambda},k \right ) \simeq \stHom_{k[u_{\lambda}]} \left ( \Omega_{k[u_{\lambda}]}^n(Au_{\lambda}),k \right ) \simeq \stHom_{k[u_{\lambda}]} \left ( Au_{\lambda},k \right )$$
The restriction map $\tau_{A, \lambda}^n(Au_{\lambda},k)$ is compatible with the isomorphisms $\Ext_A^n ( Au_{\lambda},k ) \simeq \stHom_A ( Au_{\lambda},k )$ and $\Ext_{k[u_{\lambda}]}^n ( Au_{\lambda},k ) \simeq \stHom_{k[u_{\lambda}]} ( Au_{\lambda},k )$, in the sense that the diagram
$$\xymatrix{
\Ext_A^n \left ( Au_{\lambda},k \right ) \ar[r] \ar[d]^{\tau_{A, \lambda}^n(Au_{\lambda},k)} & \stHom_A \left ( Au_{\lambda},k \right ) \ar[d]^{\tau} \\
\Ext_{k[u_{\lambda}]}^n \left ( Au_{\lambda},k \right ) \ar[r] & \stHom_{k[u_{\lambda}]} \left ( Au_{\lambda},k \right ) }$$
commutes, where the horizontal maps are the isomorphism, and $\tau$ the restriction. It therefore suffices to show that $\tau$ is injective.

The left $A$-module $Au_{\lambda}$ decomposes over $k[u_{\lambda}]$ as a direct sum $k \langle u_{\lambda} \ot e \rangle \oplus N$, where $k \langle u_{\lambda} \ot e \rangle$ denotes the $k$-vector space generated by the element $u_{\lambda} \ot e$. The latter is isomorphic to $k$ as a $k[u_{\lambda}]$-module. One now checks that the diagram
$$\xymatrix{
\stHom_A \left ( Au_{\lambda},k \right ) \ar[d]^{\tau} \ar[r] & \stHom_{k[u_{\lambda}]} \left ( k,k \right ) \ar@{^{(}->}[d] \\
\stHom_{k[u_{\lambda}]} \left ( Au_{\lambda},k \right ) \ar[r] & \stHom_{k[u_{\lambda}]} \left ( k \langle u_{\lambda} \ot e \rangle \oplus N,k \right ) }$$
commutes, where the lower horizontal map is the natural isomorphism, the upper one is the Eckmann-Shapiro isomorphism from the proof of (2) above, and the vertical map to the right is the inclusion into the summand corresponding to $k \langle u_{\lambda} \ot e \rangle$. This shows that $\tau$, and therefore also $\tau_{A, \lambda}^n(Au_{\lambda},k)$, is injective.

To prove (4), note first that $A$ decomposes as a direct sum $A = A^+ \oplus A^-$, with $A^+ = A (1 \ot (e+h))$ and $A^- = A (1 \ot (e-h))$, where $h$ is the generator of $C_2$. Similarly, one checks that $Au_{\lambda}$ decomposes as $M_{\lambda}^+ \oplus M_{\lambda}^-$, where 
$$M_{\lambda}^+ = \{ w u_{\lambda} \ot (e+h) \mid w \in \Lambda \}, \hspace{3mm} M_{\lambda}^- = \{ w u_{\lambda} \ot (e-h) \mid w \in \Lambda \}$$
As left $\Lambda$-modules, the modules $A^+$ and $A^-$ are isomorphic to $\Lambda$, hence they are indecomposable also as left $A$-modules (they represent the two isomorphism classes of indecomposable projective $A$-modules). As a consequence, the modules $M_{\lambda}^+$ and $M_{\lambda}^-$ must also be indecomposable. Now look at the exact sequence (\ref{eqn:sequence}). One checks that for $A^+$, the image of the multiplication map $\cdot (u_{\lambda} \ot e)$ is $M_{\lambda}^-$, with kernel $M_{\lambda}^+$ (and vice versa), so that $\Omega_A^1 ( M_{\lambda}^- ) = M_{\lambda}^+$. It now follows from Proposition \ref{prop:elementaryproperties} that 
$$V_A \left ( Au_{\lambda} \right ) = V_A \left ( M_{\lambda}^+ \right ) \cup V_A \left ( M_{\lambda}^- \right ) = V_A \left ( M_{\lambda}^+ \right )$$ 
and so since $M_{\lambda}^+$ is indecomposable, we see from Proposition \ref{prop:indecperiodic} that $V_A ( Au_{\lambda} )$ is irreducible.

\sloppy Let us first consider the support variety $V(Au_{\lambda},k)$, which by definition equals $Z( I_A(Au_{\lambda},k) )$, where $I_A(Au_{\lambda},k)$ is the (graded) annihilator ideal of $\Ext_A^*( Au_{\lambda},k )$ in $\Ext_A^{2*}( k,k )$. Let $\eta$ be a homogeneous element of $I_A(Au_{\lambda},k)$, and choose an element $\theta \in \Ext_A^2(Au_{\lambda},k)$ with $\tau_{A, \lambda}^*(Au_{\lambda},k)( \theta ) \neq 0$ in $\Ext_{k[u_{\lambda}]}^2 ( Au_{\lambda},k )$; this is possible by (3). Then $\eta \cdot \theta = 0$ in $\Ext_A^*( Au_{\lambda},k )$ since $\eta \in I_A(Au_{\lambda},k)$, giving
$$0 = \tau_{A, \lambda}^* \left ( Au_{\lambda},k \right ) \left ( \eta \cdot \theta \right ) = \tau_{A, \lambda}^{2*} \left ( k,k \right ) \left ( \eta \right ) \cdot \tau_{A, \lambda}^* \left ( Au_{\lambda},k \right ) \left ( \theta \right )$$
in $\Ext_{k[u_{\lambda}]}^n ( Au_{\lambda},k )$, where $\tau_{A, \lambda}^{2*} ( k,k )$ is the restriction map from $\Ext_A^{2*}(k,k)$ to $\Ext_{k[u_{\lambda}]}^{2*} ( k,k )$. We know that $\Ext_{k[u_{\lambda}]}^{2*} ( k,k )$ is just a polynomial ring of the form $k[y]$ with $y$ in degree two (see the start of the proof of Lemma \ref{lem:kernel}), and so if $\tau_{A, \lambda}^{2*} ( k,k )( \eta )$ were nonzero it would have to equal $\alpha y^t$ for some nonzero scalar $\alpha$. It is well known that multiplication by $y$ induces an isomorphism
$$\Ext_{k[u_{\lambda}]}^n ( X,k ) \xrightarrow{y \cdot} \Ext_{k[u_{\lambda}]}^{n+2} ( X,k )$$
for every $n \ge 1$ and every $k[u_{\lambda}]$-module $X$ (see, for example, \cite[pages 583-584]{PevtsovaWitherspoon}), and so since $\tau_{A, \lambda}^*(Au_{\lambda},k)( \theta ) \neq 0$, we see from the above equation that $\tau_{A, \lambda}^{2*} ( k,k )( \eta )$ cannot be nonzero in $\Ext_{k[u_{\lambda}]}^{2*} ( k,k )$. In other words, the element $\eta$ belongs to $\Ker \tau_{A, \lambda}^{2*} ( k,k )$, giving $I_A(Au_{\lambda},k) \subseteq \Ker \tau_{A, \lambda}^{2*} ( k,k )$, and then in turn
$$Z \left (  \Ker \tau_{A, \lambda}^{2*} ( k,k ) \right ) \subseteq Z \left ( I_A(Au_{\lambda},k) \right ) = V_A \left ( Au_{\lambda},k \right ) \subseteq V_A \left ( Au_{\lambda} \right ) \cap V_A \left ( k \right ) = V_A \left ( Au_{\lambda} \right )$$
where the last inclusion is Proposition \ref{prop:elementaryproperties}(2).

By definition, the support variety $V_A ( Au_{\lambda} )$ equals $Z( I_A(Au_{\lambda}) )$, where $I_A(Au_{\lambda})$ is the annihilator ideal of $\Ext_A^*( Au_{\lambda},Au_{\lambda} )$ in $\Ext_A^{2*}( k,k )$. The inclusion $Z (  \Ker \tau_{A, \lambda}^{2*} ( k,k ) ) \subseteq Z( I_A(Au_{\lambda}) )$ gives the inclusion
$$\sqrt{I_A(Au_{\lambda})} \subseteq \sqrt{\Ker \tau_{A, \lambda}^{2*} ( k,k )}$$
by \cite[Theorem 25]{Matsumura}, and so since $\Ker \tau_{A, \lambda}^{2*} ( k,k )$ is a prime ideal by Lemma \ref{lem:kernel}, we see that $I_A(Au_{\lambda}) \subseteq \Ker \tau_{A, \lambda}^{2*} ( k,k )$. We also know, from the same lemma, that $\Ker \tau_{A, \lambda}^{2*} ( k,k ) \neq \m_0$, so that the chain
$$\Ker \tau_{A, \lambda}^{2*} ( k,k ) \subset \m_0$$
of prime ideals containing $I_A(Au_{\lambda})$ has length one. Since the $A$-module $Au_{\lambda}$ is periodic, we know from Theorem \ref{thm:complexityone} that the dimension of $V_A ( Au_{\lambda} )$ is one. Moreover, we saw above that the support variety is irreducible, and so it follows that $V_A ( Au_{\lambda} ) = Z( \Ker \tau_{A,\lambda}^{2*}(k,k) )$; see the paragraph following Corollary \ref{cor:periodic}.
\end{proof}

We now use the properties we just proved for the modules $Au_{\lambda}$ to show that every non-trivial support variety contains $V_A( Au_{\lambda} )$ for some nonzero $\lambda$.

\begin{proposition}\label{prop:rankvarieties}
Let $M \in \mod A$ be a non-projective module. 

\emph{(1)} There exists a nonzero $c$-tuple $\lambda \in k^c$ with the property that $M$ is not a free module over the subalgebra $k[u_{\lambda}]$. Moreover, for every such $\lambda$, the support variety $V_A(M)$ contains the one-dimensional irreducible variety $V_A( Au_{\lambda} )$ from \emph{Proposition \ref{prop:periodictestmodule}}.

\emph{(2)} If $M$ is indecomposable and periodic, then there exists a nonzero $c$-tuple $\lambda \in k^c$ with the following property: given a nonzero $c$-tuple $\mu \in k^c$, the module $M$ is not free over $k[u_{\mu}]$ if and only if $\mu = \alpha \lambda$ for some \emph{(}necessarily nonzero\emph{)} scalar $\alpha \in k$. Moreover, $V_A(M) = V_A( Au_{\lambda} )$.
\end{proposition}

\begin{proof}
The first part of (1) follows from \cite[Section 3]{BensonErdmannHolloway}. Now take such a $c$-tuple $\lambda$. Then $\stHom_A( Au_{\lambda},M )$ is nonzero by Proposition \ref{prop:periodictestmodule}(2), and combining this with Proposition \ref{prop:periodictestmodule}(1), we obtain
$$\Ext_A^n \left ( Au_{\lambda},M \right ) \simeq \stHom_A \left ( \Omega_A^n(Au_{\lambda}),M \right ) \simeq \stHom_A \left ( Au_{\lambda},M \right ) \neq 0$$
for every $n \ge 1$. It now follows from Theorem \ref{thm:fgproperties}(7) that the support variety $V_A ( Au_{\lambda},M )$ is non-trivial, i.e.\ $V_A ( Au_{\lambda},M ) \neq \{ \m_0 \}$. The inclusion
$$V_A \left ( Au_{\lambda},M \right ) \subseteq V_A \left ( Au_{\lambda} \right ) \cap V_A \left ( M \right )$$
which holds by Proposition \ref{prop:elementaryproperties}(2), now implies that the intersection $V_A ( Au_{\lambda} ) \cap V_A ( M )$ is also non-trivial. But $V_A( Au_{\lambda} )$ is irreducible by Proposition \ref{prop:periodictestmodule}(4), and so $V_A( Au_{\lambda} ) \subseteq V_A(M)$. This proves (1).

To prove (2), suppose that $M$ is indecomposable and periodic, and let $\lambda$ be any nonzero $c$-tuple for which the module is not free over $k[u_{\lambda}]$; such a tuple exists by (1). Consider the module $M_{\lambda}^+$ from the proof of Proposition \ref{prop:periodictestmodule}(4). We showed that this module is indecomposable and periodic, and that its support variety equals that of $Au_{\lambda}$. We saw above that the intersection $V_A ( Au_{\lambda} ) \cap V_A ( M )$ is non-trivial, hence the same is trivially true for the intersection $V_A ( M_{\lambda}^+ ) \cap V_A ( M )$. It now follows from Proposition \ref{prop:indecperiodic} that $V_A(M) = V_A ( M_{\lambda}^+ ) = V_A ( Au_{\lambda} )$. 

Finally, if $\mu$ is another nonzero $c$-tuple for which $M$ is not a free $k[u_{\mu}]$-module, then what we have just shown implies that the support varieties $V_A ( Au_{\lambda} )$ and $V_A ( Au_{\mu} )$ must be equal. Then $Z( \Ker \tau_{A,\lambda}^{2*}(k,k) ) = Z( \Ker \tau_{A,\mu}^{2*}(k,k) )$ by Proposition \ref{prop:periodictestmodule}(4), giving in turn $\Ker \tau_{A,\lambda}^{2*}(k,k) = \Ker \tau_{A,\mu}^{2*}(k,k)$ since both ideals are prime ideals by Lemma \ref{lem:kernel}. The very same result gives $\mu = \alpha \lambda$ for some (nonzero) $\alpha \in k$. Conversely, if $\mu = \alpha \lambda$ for a nonzero $\alpha$, then $u_{\mu} = \alpha u_{\lambda}$. The subalgebra $k[u_{\mu}]$ then equals $k[u_{\lambda}]$, hence $M$ is not free over $k[u_{\mu}]$.
\end{proof}

\begin{remark}
The results of Propositions~\ref{prop:periodictestmodule} 
and~\ref{prop:rankvarieties} recall notions of rank varieties.
We will not use rank varieties here,
in favor of proceeding directly to the tensor product property.
However, the framework of rank varieties would be a  
natural structure in which to view our results in this section, defining 
the {\em rank variety} of an $A$-module $M$ to be the set of all
$\lambda\in k^c$ for which $M$ is not free as a $k[u_{\lambda}]$-module.
Equivalently, by Proposition~\ref{prop:periodictestmodule}(2),
this is the set of all $\lambda\in k^c$ such that
$\underline{\Hom}_A(Au_{\lambda},M)\neq 0$. 
By Proposition~\ref{prop:rankvarieties}(2), if $M$ is
indecomposable and periodic, then this set is simply the line through
the $c$-tuple $\lambda$ in the statement.
Compare with~\cite{BensonErdmannHolloway,BerghErdmann,PevtsovaWitherspoon},
with~\cite{DK,DS,GHSS} in characteristic~0, and with~\cite{BIKP} in odd
characteristic. 
\end{remark}

In the main result of this section, we consider 
general braided Hopf algebras of the form $\Lambda \rtimes G$, 
for $G$ a finite group containing $C_2$, and over an 
algebraically closed field $k$. 
We know from Remark \ref{rem:fg}(1) that when the characteristic 
of $k$ does not divide the order of $G$, then the finite tensor 
categories $\mod ( \Lambda \rtimes G)$ and $\mod ( \Lambda \rtimes C_2)$ 
satisfy \textbf{Fg}. 
The following lemma allows us to pass from support varieties 
over $\Lambda \rtimes G$ to support varieties over $\Lambda \rtimes C_2$.

\begin{lemma}\label{lem:reducingvarieties}
Let $k$ be an algebraically closed field, $c$ a positive integer, and $\Lambda$ the exterior algebra on $c$ generators over $k$. Furthermore, let $G$ be a finite group whose order is not divisible by the characteristic of $k$, acting on $\Lambda$ in such a way that the algebra $H = \Lambda \rtimes G$ is a Hopf algebra. Finally, suppose that $G$ contains a central subgroup $C_2$ of order two, acting on $\Lambda$ by letting its generator change the sign of the generators of $\Lambda$, and that $A = \Lambda \rtimes C_2$ is a Hopf subalgebra of $H$. Then 
$$V_H(M) = V_H(N) \hspace{2mm} \Longrightarrow \hspace{2mm} V_A(M) = V_A(N)$$
for all $M,N \in \mod H$.
\end{lemma}

\begin{proof}
We know from Remark \ref{rem:fg}(3) that the cohomology ring $\Ext_H^*(k,k)$ is isomorphic to $\Ext_{\Lambda}^*(k,k)^{G}$ via the restriction map
$$\Ext_H^{*}(k,k) \xrightarrow{\tau_{H, \Lambda}^*(k,k)} \Ext_{\Lambda}^{*}(k,k)$$
This map is the composite
$$\Ext_H^{*}(k,k) \xrightarrow{\tau_{H, A}^*(k,k)} \Ext_{A}^{*}(k,k) \xrightarrow{\tau_{A, \Lambda}^*(k,k)} \Ext_{\Lambda}^{*}(k,k)$$
and from Remark \ref{rem:evendegrees} we also know that
$$\Ext_{A}^{2*}(k,k) \xrightarrow{\tau_{A, \Lambda}^{2*}(k,k)} \Ext_{\Lambda}^{2*}(k,k)$$
is an isomorphism. Since $C_2 \subseteq G$, both $\Ext_H^{*}(k,k)$ and $\Ext_A^{*}(k,k)$ are concentrated in even degrees. 

By definition, the action of $G$ on $\Lambda$ is defined in terms of a group homomorphism $G \to \Aut ( \Lambda )$. Now for an element $a = w_1 \ot e + w_2 \ot h$ in $A$ and $g \in G$, we define ${^ga}$ to be ${^gw_1} \ot e + {^gw_2} \ot h$. One checks that this induces an automorphism of $A$, using the fact that $C_2$ is central in $G$. Moreover, in this way we obtain a homomorphism $G \to \Aut (A)$, with the action of $G$ on $A$ extending the action on $\Lambda$. As in Remark \ref{rem:fg}(2), we obtain a $G$-action on $\Ext_{A}^{*}(k,k)$, and this action commutes with the restriction map (and isomorphism)
$$\Ext_{A}^{2*}(k,k) \xrightarrow{\tau_{A, \Lambda}^{2*}(k,k)} \Ext_{\Lambda}^{2*}(k,k)$$
Then $\Ext_{\Lambda}^{2*}(k,k)^G$ is the image of $\Ext_{A}^{2*}(k,k)^G$, and so  $\Ext_H^{2*}(k,k)$ is isomorphic to $\Ext_{A}^{2*}(k,k)^{G}$ via the restriction map
$$\Ext_H^{2*}(k,k) \xrightarrow{\tau_{H, A}^{2*}(k,k)} \Ext_{A}^{2*}(k,k)$$
in light of the above.

Let $M$ and $N$ be $H$-modules with $V_H(M) = V_H(N)$. There is a  commutative diagram
$$\xymatrix@C=50pt{
\Ext_H^{2*}(k,k) \ar[d]^{\varphi_M^H} \ar[r]^{\tau_{H,A}^{2*}} & \Ext_{A}^{2*}(k,k) \ar[d]^{\varphi_M^A} \\
\Ext_{H}^* (M,M) \ar[r]^{\tau_{H, A}^{*}(M,M)} & \Ext_{A}^* (M,M)  }$$
where the horizontal maps are restrictions (we have skipped the arguments in the upper one, since we shall be using it quite a lot in what follows), and the vertical maps are induced by tensoring with $M$. The module $N$ gives rise to a similar diagram. By Lemma \ref{lem:restriction}, the horizontal restriction maps are injective. Denote by $I_A(M)$ the annihilator ideal of $\Ext_{A}^* (M,M)$ in $ \Ext_{A}^{2*}(k,k)$, that is, $I_A(M) = \Ker \varphi_M^A$, and similarly for $I_A(N), I_H(M)$ and $I_H(N)$. These are the ideals defining the four support varieties we are considering.

Suppose we can show that $I_A(M)$ and $I_A(N)$ are $G$-invariant in $\Ext_{A}^{2*}(k,k)$, so that ${^gI_A(M)} = I_A(M)$ and ${^gI_A(N)} = I_A(N)$ for all $g \in G$. Let $\m \in V_A(M)$; thus $\m$ is a maximal ideal of $\Ext_{A}^{2*}(k,k)$ with $I_A(M) \subseteq \m$. Since $k$ is algebraically closed and the algebras $\Ext_H^{2*}(k,k)$ and $\Ext_A^{2*}(k,k)$ are finitely generated, the ideal $( \tau_{H,A}^{2*} )^{-1} ( \m )$ is maximal in $\Ext_H^{2*}(k,k)$; see, for example, \cite[Section 5.4]{Benson}. The commutativity of the diagram gives $I_H(M) \subseteq ( \tau_{H,A}^{2*} )^{-1} ( \m )$, and therefore $( \tau_{H,A}^{2*} )^{-1} ( \m ) \in V_H(M)$. Suppose, on the other hand, that $\m \notin V_A(N)$, so that $I_A(N)\nsubseteq \m$. As $I_A(N)$ is $G$-invariant, this gives $I_A(N) \nsubseteq {^g\m}$ for every $g \in G$, and so by prime avoidance there exists a homogeneous element $\eta \in I_A(N)$ with $\eta \notin {^g\m}$ for every $g \in G$. Consider now the element
$$w = \prod_{g \in G} {^g\eta}$$
It belongs to $I_A(N)$ since $\eta$ is one of the factors, but it cannot belong to $\m$; if it did, then ${^g\eta}$ would belong to $\m$ for some $g$, giving $\eta \in {^{g^{-1}}\m}$. Furthermore, this element is $G$-invariant, and therefore belongs to the image of $\tau_{H,A}^{2*}(k,k)$, i.e.\ $w = \tau_{H,A}^{2*} ( \theta )$ for some $\theta \in \Ext_H^{2*}(k,k)$. 

The commutativity of the diagram with $M$ replaced by $N$ gives 
$$\tau_{H, A}^{*}(N,N) \circ \varphi_N^H ( \theta ) = \varphi_N^A \circ \tau_{H,A}^{2*} ( \theta ) = \varphi_N^A(w) = 0$$
since $w \in I_A(N)$, and so since $\tau_{H, A}^{*}(N,N)$ is injective we obtain $\theta \in I_H(N)$. Now $\theta$ does not belong to $( \tau_{H,A}^{2*} )^{-1} ( \m )$, for if it did, then $w$ would be an element of $\m$. Therefore $I_H(N) \nsubseteq ( \tau_{H,A}^{2*} )^{-1} ( \m )$, so that $( \tau_{H,A}^{2*} )^{-1} ( \m ) \notin V_H(N)$. But $( \tau_{H,A}^{2*} )^{-1} ( \m ) \in V_H(M)$ from above, and $V_H(M) = V_H(N)$ by assumption, and so we have reached a contradiction. It must therefore be the case that $\m \in V_A(N)$, giving $V_A(M) \subseteq V_A(N)$. The reverse inclusion is proved similarly, hence $V_A(M) =  V_A(N)$.

It only remains to show that the ideals $I_A(M)$ and $I_A(N)$ are $G$-invariant in $\Ext_{A}^{2*}(k,k)$. We prove this for $I_A(M)$; the proof for $I_A(N)$ is similar. It follows from \cite[Theorem 9.3.9]{Witherspoon} that $I_A(M)$ equals the annihilator ideal of the $\Ext_{A}^{2*}(k,k)$-module $\Ext_A^*(k, M^* \ot M)$, where the module action is given in terms of Yoneda composition. Now let $\eta$ and $\theta$ be homogeneous elements of $I_A(M)$ and $\Ext_A^*(k, M^* \ot M)$, respectively. Given any $H$-module $X$ and an element $g \in G$, the twisted $A$-module ${_gX}$ is isomorphic to $X$, with an isomorphism ${_gX} \to X$ mapping an element $m$ to $(1 \ot g^{-1})m$. Consequently, when we twist a homogeneous element of $\Ext_A^*(k, M^* \ot M)$ by an element from $G$, we obtain a new element in $\Ext_A^*(k, M^* \ot M)$, since $k$ and $M^* \ot M$ are $H$-modules. Therefore, as $\eta$ belongs to $I_A(M)$, we obtain
$$\theta \circ ( {^g\eta} ) = ( {^{gg^{-1}}\theta} ) \circ ( {^g\eta} ) = {^g \left ( ( {^{g^{-1}}\theta} ) \circ \eta \right )} = 0$$
since ${^{g^{-1}}\theta}$ belongs to $\Ext_A^*(k, M^* \ot M)$. This shows that ${^g\eta} \in I_A(M)$, and hence $I_A(M)$ is $G$-invariant.
\end{proof}

We now prove the main result of this section: the tensor product property holds for support varieties over braided Hopf algebras of the form we have been considering.

\begin{theorem}\label{thm:skewgroupalgs}
Let $k$ be an algebraically closed field, $c$ a positive integer, and $\Lambda$ the exterior algebra on $c$ generators over $k$. Furthermore, let $G$ be a finite group whose order is not divisible by the characteristic of $k$, acting on $\Lambda$ in such a way that the algebra $H = \Lambda \rtimes G$ is a braided Hopf algebra. Finally, suppose that $G$ contains a central subgroup $C_2$ of order two, acting on $\Lambda$ by letting its generator change the sign of the generators of $\Lambda$, and that $\Lambda \rtimes C_2$ is a Hopf subalgebra of $H$. Then 
$$V_H \left ( M \ot N \right ) = V_H \left ( M \right ) \cap V_H \left ( N \right )$$
for all $M,N \in \mod H$.
\end{theorem}

\begin{proof}
As before, denote by $A$ the Hopf subalgebra $\Lambda \rtimes C_2$ of $H$. Let $M$ and $N$ be two nonzero periodic $H$-modules with $V_H(M) = V_H(N)$, and decompose them as $A$-modules into direct sums $M = \oplus M_i$, $N = \oplus_j N_j$ of indecomposable modules. Since $H$ is free as a left $A$-module (see \cite[Theorem 7]{NZ}), both $M$ and $N$ are of complexity at most one over $A$, because the projective resolutions over $H$ restrict to projective resolutions over $A$. Moreover, the modules cannot be projective over $A$; if $M$, say, is $A$-projective, then it is also projective -- and hence free -- over $\Lambda$, since $A$ is free over $\Lambda$. Then we would obtain a free module when we induced $M$ (as a $\Lambda$-module) back to $H$, but as in the proof of Lemma \ref{lem:restriction}, the original $H$-module $M$ is a summand of this induced module. As $M$ is not projective over $H$, it must be the case that it is not projective over $A$ either. Therefore both $M$ and $N$ are of complexity one over $A$. In particular, at least one of the $M_i$, and one of the $N_j$, is not projective, and therefore periodic from Corollary \ref{cor:periodic}.

By Lemma \ref{lem:reducingvarieties} there is an equality $V_A(M) = V_A(N)$, and by Theorem \ref{thm:fgproperties}(2) these support varieties are non-trivial since $M$ and $N$ are not projective over $A$. Consequently, by Proposition \ref{prop:elementaryproperties}(1), there exist indices $i$ and $j$ for which $V_A(M_i) \cap V_A(N_j) \neq \{ \m_0 \}$, where $\m_0$ is the graded maximal ideal of $\Ext_{A}^{2*}(k,k)$. Using Theorem \ref{thm:fgproperties}(2) again, we see that $M_i$ and $N_j$ are not projective, and therefore periodic from the above. It now follows from Proposition \ref{prop:indecperiodic} that $V_A(M_i) = V_A(N_j)$, and so from Proposition \ref{prop:rankvarieties} we see that there exists a nonzero $c$-tuple $\lambda \in k^c$ with $V_A(M_i) = V_A(N_j) = V_A(Au_{\lambda})$, and with $M_i$ and $N_j$ not free over the subalgebra $k[u_{\lambda}]$ of $A$. Then $M$ and $N$ are not free over $k[u_{\lambda}]$, either.

Since $u_{\lambda}$ is just a linear combination of the elements $x_1, \dots, x_c \in \Lambda$, the group $C_2$ acts on $k[u_{\lambda}]$, and we may form the four-dimensional skew group algebra $H_4^{\lambda} = k[u_{\lambda}] \rtimes C_2$. This is a Hopf subalgebra of $A$ (and therefore also of $H$), isomorphic to the Sweedler Hopf algebra $H_4$, and it contains $k[u_{\lambda}]$ as a subalgebra. Since it is free over $k[u_{\lambda}]$, the modules $M$ and $N$ cannot be projective as $H_4^{\lambda}$-modules, for if they were, then they would also be free over $k[u_{\lambda}]$. 

The algebra $H_4^{\lambda}$ has two simple modules, namely the trivial module $k$ and a module $S$. The latter is one-dimensional, with $u_{\lambda}S =0$, and $h$ acting as $-1$ (we identify $H_4^{\lambda}$ with a $k$-algebra with basis $1, u_{\lambda}, h$ and $hu_{\lambda}$, where $h$ is the generator of $C_2$). It is well-known that these are the only non-projective indecomposable $H_4^{\lambda}$-modules (see, for example, \cite[Page 467]{Cibils} or \cite[Corollary 2.4 and Theorem 2.5]{COZ}), and so it follows that there are elements $m \in M$ and $n \in N$ that generate summands isomorphic to either $k$ or $S$ when we restrict $M$ and $N$ to $H_4^{\lambda}$. Let $W$ be the one-dimensional subspace of $M \ot N$ generated by $m \ot n$. This is an $H_4^{\lambda}$-submodule of $M \ot N$; the comultiplication on $H_4^{\lambda}$ maps $u_{\lambda}$ to $u_{\lambda} \ot 1 + h \ot u_{\lambda}$ and $h$ to $h \ot h$, so that $u_{\lambda}$ acts as zero on $W$, and $h$ as $1$ or $-1$. Therefore, over $H_4^{\lambda}$, the module $M \ot N$ has a direct summand isomorphic to either $k$ or $S$. In particular, $M \ot N$ is not projective as an $H_4^{\lambda}$-module. Now since $H_4^{\lambda}$ is a Hopf subalgebra of $H$, we know from \cite[Theorem 7]{NZ} that $H$ is free as an $H_4^{\lambda}$-module. This implies that $M \ot N$ cannot be projective over $H$, for it it were, then it would also be projective over $H_4^{\lambda}$.

We have shown that for every pair of nonzero periodic $H$-modules whose support varieties coincide, the tensor product is not projective. It therefore follows from Theorem \ref{thm:mainalternative} that $V_H(M \ot N) = V_H(M) \cap V_H(N)$ for all $H$-modules $M$ and $N$.
\end{proof}

By Deligne's famous classification theorem (see \cite{Deligne2}), every symmetric finite tensor category over an algebraically closed field of characteristic zero is equivalent to the category of finite dimensional representations of some affine supergroup scheme. This means precisely that such a category is equivalent to $\mod H$, where $H$ is a Hopf algebra of the form $\Lambda \rtimes G$ for some exterior algebra $\Lambda$ and finite group $G$. Furthermore, there is a subgroup of $G$ of order two, and all the assumptions in Theorem \ref{thm:skewgroupalgs} are satisfied (see \cite{AndruskiewitschEtingofGelaki} and also \cite[Section 7.1]{NP}). Thus we obtain the following corollary of Theorem~\ref{thm:skewgroupalgs}, giving a different approach to~\cite[Corollary 3.2.4]{DK}. 

\begin{corollary}\label{thm:mainsymmetric}
Suppose that $( \C, \ot, \unit )$ is a symmetric finite tensor category over an algebraically closed field of characteristic zero. Then
$$\VC(X \ot Y) = \VC(X) \cap \VC(Y)$$
for all objects $X,Y \in \C$.
\end{corollary}



\end{document}